\documentclass[11pt,reqno]{amsart}

\usepackage{amsmath,amsfonts,amsthm,amssymb}
\usepackage{graphicx}
\usepackage{verbatim}
\usepackage{color}
\usepackage{amscd}
\usepackage{verbatim}
\usepackage{mathrsfs}
\usepackage[colorlinks,citecolor=red,pagebackref,hypertexnames=false, breaklinks]{hyperref}

\begin{document}
\newcommand{\M}{{\mathcal M}}
\newcommand{\loc}{{\mathrm{loc}}}
\newcommand{\core}{C_0^{\infty}(\Omega)}
\newcommand{\sob}{W^{1,p}(\Omega)}
\newcommand{\sobloc}{W^{1,p}_{\mathrm{loc}}(\Omega)}
\newcommand{\merhav}{{\mathcal D}^{1,p}}
\newcommand{\be}{\begin{equation}}
\newcommand{\ee}{\end{equation}}
\newcommand{\mysection}[1]{\section{#1}\setcounter{equation}{0}}
\newcommand{\laplace}{\Delta}
\newcommand{\pl}{\laplace_p}
\newcommand{\grad}{\nabla}
\newcommand{\pd}{\partial}
\newcommand{\bo}{\pd}
\newcommand{\csub}{\subset \subset}
\newcommand{\sm}{\setminus}
\newcommand{\ssm}{:}
\newcommand{\diver}{\mathrm{div}\,}
\newcommand{\bea}{\begin{eqnarray}}
\newcommand{\eea}{\end{eqnarray}}
\newcommand{\bean}{\begin{eqnarray*}}
\newcommand{\eean}{\end{eqnarray*}}
\newcommand{\thkl}{\rule[-.5mm]{.3mm}{3mm}}
\newcommand{\cw}{\stackrel{\rightharpoonup}{\rightharpoonup}}
\newcommand{\id}{\operatorname{id}}
\newcommand{\supp}{\operatorname{supp}}
\newcommand{\wlim}{\mbox{ w-lim }}
\newcommand{\mymu}{{x_N^{-p_*}}}
\newcommand{\R}{{\mathbb R}}
\newcommand{\N}{{\mathbb N}}
\newcommand{\Z}{{\mathbb Z}}
\newcommand{\Q}{{\mathbb Q}}
\newcommand{\abs}[1]{\lvert#1\rvert}
\newtheorem{theorem}{Theorem}[section]
\newtheorem{corollary}[theorem]{Corollary}
\newtheorem{lemma}[theorem]{Lemma}
\newtheorem{notation}[theorem]{Notation}
\newtheorem{definition}[theorem]{Definition}
\newtheorem{remark}[theorem]{Remark}
\newtheorem{proposition}[theorem]{Proposition}
\newtheorem{assertion}[theorem]{Assertion}
\newtheorem{problem}[theorem]{Problem}
\newtheorem{conjecture}[theorem]{Conjecture}
\newtheorem{question}[theorem]{Question}
\newtheorem{example}[theorem]{Example}
\newtheorem{Thm}[theorem]{Theorem}
\newtheorem{Lem}[theorem]{Lemma}
\newtheorem{Pro}[theorem]{Proposition}
\newtheorem{Def}[theorem]{Definition}
\newtheorem{Exa}[theorem]{Example}
\newtheorem{Exs}[theorem]{Examples}
\newtheorem{Rems}[theorem]{Remarks}
\newtheorem{Rem}[theorem]{Remark}

\newtheorem{Cor}[theorem]{Corollary}
\newtheorem{Conj}[theorem]{Conjecture}
\newtheorem{Prob}[theorem]{Problem}
\newtheorem{Ques}[theorem]{Question}
\newtheorem*{corollary*}{Corollary}
\newtheorem*{theorem*}{Theorem}
\newcommand{\pf}{\noindent \mbox{{\bf Proof}: }}


\renewcommand{\theequation}{\arabic{equation}}
\catcode`@=11 \@addtoreset{equation}{section} \catcode`@=12
\newcommand{\Real}{\mathbb{R}}
\newcommand{\real}{\mathbb{R}}
\newcommand{\Nat}{\mathbb{N}}
\newcommand{\ZZ}{\mathbb{Z}}
\newcommand{\CC}{\mathbb{C}}
\newcommand{\Pess}{\opname{Pess}}
\newcommand{\Proof}{\mbox{\noindent {\bf Proof} \hspace{2mm}}}
\newcommand{\mbinom}[2]{\left (\!\!{\renewcommand{\arraystretch}{0.5}
\mbox{$\begin{array}[c]{c}  #1\\ #2  \end{array}$}}\!\! \right )}
\newcommand{\brang}[1]{\langle #1 \rangle}
\newcommand{\vstrut}[1]{\rule{0mm}{#1mm}}
\newcommand{\rec}[1]{\frac{1}{#1}}
\newcommand{\set}[1]{\{#1\}}
\newcommand{\dist}[2]{$\mbox{\rm dist}\,(#1,#2)$}
\newcommand{\opname}[1]{\mbox{\rm #1}\,}
\newcommand{\mb}[1]{\;\mbox{ #1 }\;}
\newcommand{\undersym}[2]
 {{\renewcommand{\arraystretch}{0.5}  \mbox{$\begin{array}[t]{c}
 #1\\ #2  \end{array}$}}}
\newlength{\wex}  \newlength{\hex}
\newcommand{\understack}[3]{%
 \settowidth{\wex}{\mbox{$#3$}} \settoheight{\hex}{\mbox{$#1$}}
 \hspace{\wex}  \raisebox{-1.2\hex}{\makebox[-\wex][c]{$#2$}}
 \makebox[\wex][c]{$#1$}   }%
\newcommand{\smit}[1]{\mbox{\small \it #1}}
\newcommand{\lgit}[1]{\mbox{\large \it #1}}
\newcommand{\scts}[1]{\scriptstyle #1}
\newcommand{\scss}[1]{\scriptscriptstyle #1}
\newcommand{\txts}[1]{\textstyle #1}
\newcommand{\dsps}[1]{\displaystyle #1}
\newcommand{\dx}{\,\mathrm{d}x}
\newcommand{\dy}{\,\mathrm{d}y}
\newcommand{\dz}{\,\mathrm{d}z}
\newcommand{\dt}{\,\,dt}
\newcommand{\dr}{\,\mathrm{d}r}
\newcommand{\du}{\,\mathrm{d}u}
\newcommand{\dv}{\,\mathrm{d}v}
\newcommand{\dV}{\,\mathrm{d}V}
\newcommand{\ds}{\,\mathrm{d}s}
\newcommand{\dS}{\,\mathrm{d}S}
\newcommand{\dk}{\,\mathrm{d}k}

\newcommand{\dphi}{\,\mathrm{d}\phi}
\newcommand{\dtau}{\,\mathrm{d}\tau}
\newcommand{\dxi}{\,\mathrm{d}\xi}
\newcommand{\deta}{\,\mathrm{d}\eta}
\newcommand{\dsigma}{\,\mathrm{d}\sigma}
\newcommand{\dtheta}{\,\mathrm{d}\theta}
\newcommand{\dnu}{\,\mathrm{d}{\nu'}}

\def\ga{\alpha}     \def\gb{\beta}       \def\gg{\gamma}
\def\gc{\chi}       \def\gd{\delta}      \def\ge{\varepsilon}
\def\gth{\theta}                         \def\vge{\varepsilon}
\def\gf{\phi}       \def\vgf{\varphi}    \def\gh{\eta}
\def\gi{\iota}      \def\gk{\nu}      \def\gl{\lambda}
\def\gm{\mu}        \def\gn{{\nu'}}         \def\gp{\pi}
\def\vgp{\varpi}    \def\gr{\rho}        \def\vgr{\varrho}
\def\gs{\sigma}     \def\vgs{\varsigma}  \def\gt{\tau}
\def\gu{\upsilon}   \def\gv{\vartheta}   \def\gw{\omega}
\def\gx{\xi}        \def\gy{\psi}        \def\gz{\zeta}
\def\Gg{\Gamma}     \def\Gd{\Delta}      \def\Gf{\Phi}
\def\Gth{\Theta}
\def\Gl{\Lambda}    \def\Gs{\Sigma}      \def\Gp{\Pi}
\def\Gw{\Omega}     \def\Gx{\Xi}         \def\Gy{\Psi}

\renewcommand{\div}{\mathrm{div}}
\newcommand{\red}[1]{{\color{red} #1}}

\newcommand{\cqfd}{\begin{flushright}                  
			 $\Box$
                 \end{flushright}}
\newcommand{\Ric}{\mathrm{Ric}}

\newcommand{\D}{\vec{\Delta}}

\numberwithin{theorem}{section}
\numberwithin{equation}{section}


\title{Gaussian heat kernel estimates: from functions to forms}

\author{Thierry Coulhon}
\address{Thierry Coulhon, PSL Research University, 75005 Paris, France}
\email{thierry.coulhon@univ-psl.fr}

\author{Baptiste Devyver}
\address{Baptiste Devyver, Department of Mathematics, Technion, 32000 Haifa, Israel}
\email{devyver@tx.technion.ac.il}

\author{Adam Sikora}
\address{Adam Sikora, Department of Mathematics, Macquarie
University, NSW 2109, Australia}
\email{adam.sikora@mq.edu.au}


\date{\today}


\begin{abstract}  On a complete non-compact Riemannian manifold satisfying the volume doubling property, we give conditions on the negative part of the Ricci curvature that
ensure that, unless there are  harmonic one-forms, the Gaussian heat kernel upper estimate on functions transfers to one-forms. These conditions do no entail any constraint on the size of the Ricci curvature, only on its decay at infinity.
\end{abstract}

\maketitle
\tableofcontents

\section{Introduction and statement of the results}

\subsection{Preliminaries}

On Riemannian manifolds and more general spaces, heat kernel estimates on functions are fairly well understood (see for instance the surveys \cite{C-s}, \cite{G4}, \cite{SC-s}  and the books  \cite{SC}, \cite{Gbook}).  In short, a uniform upper estimate for the heat kernel on functions is equivalent to an $L^2$ isoperimetric-type functional inequality (see \cite{C-s}).  In the setting of  Riemannian manifolds with the volume doubling property,  the Gaussian upper estimate of the heat kernel  on functions can be characterised in terms of relative Faber-Krahn inequalities and other Sobolev type inequalities  (see \cite{G2}, \cite{SC}, \cite{Gbook},  as well as the more recent \cite{BCS}). By contrast, heat kernel estimates for the Hodge Laplacian acting on forms have been traditionally considered only for small time and on compact manifolds (\cite{BGV}, \cite{Gi}, \cite{R}) or in  special  settings like Lie groups (\cite{Ru}, \cite{MPR}). The main difficulty here is that one cannot rely on the positivity of the heat kernel, or equivalently on the maximum principle. More recently, the question of finding Gaussian estimates valid for all time for the heat kernel acting on forms on non-compact manifolds has been addressed in \cite{CZ} and later on in \cite{D2}. The most up-to-date result, from \cite{D2}, yields Gaussian upper estimates for the heat kernel on $1$-forms if the manifold satisfies some geometric conditions (Sobolev inequality of dimension $n>4$, Euclidean volume growth,  the negative part of the Ricci curvature belonging to $L^{\frac{n}{2}\pm\varepsilon}$, $\varepsilon>0$), and in absence of harmonic $L^2$ forms, the latter condition being  necessary. Our goal in this article is to extend  the results of \cite{D2} to a more general class of manifolds, which do not necessarily satisfy a global Sobolev inequality. More precisely, we shall replace Sobolev by the more natural assumption of a Gaussian upper estimate for the heat kernel acting on functions. As a consequence, we will not be restricted to manifolds with uniform polynomial volume growth, and our results will hold for manifolds with the volume doubling property. Multiple difficulties appear when one tries to do this, and new ideas and techniques have to be introduced on top of those from  \cite{D2}. Before presenting our results, let us first introduce the setting.

Let $M$ be a complete, connected, non-compact Riemannian manifold, endowed with a positive measure $\mu=e^{f}\nu$ that is absolutely continuous with respect to the Riemannian measure $\nu$. We assume that $f$ is smooth. We denote by $\nabla$ the Riemannian gradient and by $\Delta_\mu$  the weighted non-negative Laplace operator defined by $\Delta_\mu u=-\div(\nabla u)-\langle \nabla f,\nabla u\rangle$. In the sequel, we will denote $\Delta_\mu$ simply by $\Delta$, the dependance on the measure $\mu$ being  implicit. Let $d$ denote the geodesic distance, $B(x,r)$ the open  ball for  $d$ with centre $x\in M$ and radius $r>0$, and $V(x,r)$ its volume $\mu\left(B(x,r)\right)$.

We will use the notation $h\lesssim g$ to indicate that there exists a constant $C$ (independent of the important parameters) such that $h\leq Cg$, $h\gtrsim g$ if $g\lesssim h$,  and $h\simeq g$ if $h\lesssim g$ and $h\gtrsim g$.

We will assume that the weighted manifold $(M,d,\mu)$ satisfies the volume doubling property, that is
  \begin{equation}\label{d}\tag{$V\!D$}
     V(x,2r)\lesssim  V(x,r),\quad \forall~x \in M,~r > 0.
    \end{equation}
It follows easily that there exists  $\nu>0$  such that
     \begin{equation*}\label{dnu}\tag{$V\!D_\nu$}
      \frac{V(x,r)}{V(x,s)}\lesssim \left(\frac{r}{s}\right)^{\nu} ,\quad \forall~ x \in M,~r \geq s>0.
    \end{equation*}
It is known (see  \cite[Theorem 1.1]{G1}) that if $M$ is  connected, non-compact, and satisfies \eqref{d}, then the following reverse doubling condition holds:
    there
exist
 $0<\nu'\leq \nu$  such that,
for all $r\geq s>0$ and $x\in M$,
\begin{equation}\label{rnu}\tag{$RD_{\nu'}$}
 \frac{V(x,r)}{V(x,s)}\gtrsim\left(\frac{r}{s}\right)^{\nu'}.
\end{equation}
Let us introduce the following volume lower bound: for some $x_0\in M$  there exists  $\kappa>0$ 
such that   
  \begin{equation}\label{dd}\tag{$V\!L_{\kappa}$}
  V(x_0,r)\gtrsim  r^{\kappa},\qquad\forall r\geq1.
  \end{equation}
It is  clear from \eqref{d}  that  this condition does not depend on the choice of 
$x_0$.  If \eqref{rnu} holds, taking $s=1$  shows that \eqref{dd}  follows at least for  $\kappa=\nu'$.


Let $e^{-t\Delta}$ be the heat operator associated to $\Delta$, and  $p_t(x,y)$ its kernel, so that, for any  compactly supported smooth function $f$ on $M$, there holds:
$$e^{-t\Delta}u(x)=\int_M p_t(x,y)u(y)d\mu(y).$$
It is classical that $p_t$ is smooth, positive, satisfies $p_t(x,y)=p_t(y,x)$, and that under \eqref{d}, $\int_Mp_t(x,y)\,d\mu(y)= 1$ for all $x$, in other words $M$ is stochastically  complete. 

On and off-diagonal estimates of the heat kernel $p_t(x,y)$ have been studied in detail in the past thirty years. Let us start by introducing the on-diagonal estimate:

\begin{equation}\tag{$DU\!E$}
p_{t}(x,x)\lesssim
\frac{1}{V(x,\sqrt{t})}, \quad \forall~t>0,\,\forall\,x\in
 M. \label{due}
\end{equation}
Under  \eqref{d}, \eqref{due} self-improves into an (off-diagonal) Gaussian upper estimate (\cite[Theorem 1.1]{Gr1}, see also \cite[Section 4.2]{CS}):
\begin{equation}\tag{$U\!E$}
p_{t}(x,y)\lesssim
\frac{1}{V(x,\sqrt{t})}\exp
\left(-\frac{d^{2}(x,y)}{Ct}\right), \quad \forall~t>0,\, \mbox{a.e. }x,y\in
 M,\label{UE}
\end{equation}
for some $C>0$. It is known (see \cite{G2}) that the Gaussian upper estimate \eqref{UE} is equivalent to an $L^2$  isoperimetric-type inequality called the relative Faber-Krahn inequality (for a new point of view on this equivalence see \cite{BCS}). Let us also introduce the upper and lower Gaussian estimates for the heat kernel (sometimes called Li-Yau estimates):
\begin{equation}\tag{$LY$}
p_{t}(x,y)\simeq
\frac{1}{V(x,\sqrt{t})}\exp
\left(-\frac{d^{2}(x,y)}{Ct}\right), \quad \forall~t>0,\, \mbox{a.e. }x,y\in
 M,\label{LY}
\end{equation}
where $C$ denotes a possibly different constant in the upper and the lower bound.
it was proved in \cite{LY} that \eqref{LY} holds on a complete manifold with non-negative Ricci curvature endowed with its Riemannian measure.
By a theorem of  Saloff-Coste (see \cite{SC} and the references therein, see also \cite{G1} for an alternative approach of the main implication, as well as the more recent \cite{BCF1}), it is known that under \eqref{d}, \eqref{LY} is equivalent to the following family of scale-invariant $L^2$ Poincar\'{e} inequalities: there is a constant $C$ such that, for every geodesic ball $B=B(x,r)$, and every $u\in C^\infty(B)$,
\begin{equation}\label{P}\tag{$P$}
\int_{B}|u-u_{B}|^2\,d\mu\leq Cr^2\int_{B}|\nabla u|^2\,d\mu,
\end{equation}
where $u_{B}=\frac{1}{\mu(B)}\int_Bu\,d\mu$ denotes the average of $u$ over $B$.  In Remark \ref{nogent} below, we shall also use a weak local $L^1$ version of \eqref{P}: we say that \eqref{P_loc} holds if
\begin{equation}\label{P_loc}\tag{$P_{loc}$}
\int_{B}|u-u_{B}|\,d\mu\leq C_r\int_{B}|\nabla u|\,d\mu,
\end{equation}
where this time the constant $C_{r}$ depends on the radius of $B$ (but not on its center), and is not necessarily linear in $r$. Notice that by standard arguments involving H\"{o}lder inequality, the $L^1$ local Poincar\'{e} inequality \eqref{P_loc} implies its $L^p$ counterpart. It is classical that \eqref{P_loc} holds on a complete manifold with Ricci curvature bounded from below, endowed with its Riemannian measure (see \cite{B}).

%
%
We now recall the notion of non-parabolicity (see for instance \cite[Section 5]{G4} for more information).  One says that $M$ is non-parabolic if
$$\int_1^\infty p_t(x,y)\,dt<+\infty$$
for some (all) $x, y\in M$.  In this case, $G(x,y)$ defined by $$G(x,y)=\int_0^{+\infty} p_t(x,y)\,dt$$ is finite for all $x\neq y$, and is the positive, minimal Green function of $\Delta$.
If furthermore \eqref{d} and \eqref{UE} hold, then the non-parabolicity of $M$ is equivalent to
\begin{equation}\label{vol1}
\int_1^{+\infty}\frac{dt}{V(x_0, \sqrt{t})}<+\infty,
\tag{$V^\infty$}
\end{equation}
for some  $x_0\in M$ (this uses the fact that under \eqref{d} and \eqref{UE}, the heat kernel has an on-diagonal lower bound $p_t(x,x)\geq \frac{C}{V(x,\sqrt{t})}$; see \cite[Theorem 11.1]{G4}). We also introduce a more general integral volume growth condition:
\begin{equation}\label{vol2}
\int_1^{+\infty}\frac{dt}{\left[V(x_0, \sqrt{t})\right]^{1-\frac{1}{p}}}<+\infty\tag{$V^p$}
\end{equation}
for some  $x_0\in M$ and $p\in (1,+\infty]$.
It follows from condition  \eqref{d} that the validity of conditions \eqref{vol2} does not depend on the 
choice of $x_0$. Notice that \eqref{dd} with $\kappa>2$ implies \eqref{vol2} for all $p\in \left(\frac{\kappa}{\kappa-2},+\infty\right]$, hence non-parabolicity under  \eqref{d} and \eqref{UE}.

In this article, we investigate Gaussian estimates for the heat kernel of elliptic operators of Schr\"{o}dinger type, acting on sections of a vector bundle over $M$. We are motivated by the particular case of the Hodge Laplacian $\vec{\Delta}_\mu=dd_\mu^*+d_\mu^*d$, acting on $1$-forms, which we describe now. In this paragraph, in order to make things clearer, we assume the measure $\mu$ to be the Riemannian measure, although, as we will see later, what we are going to say can be extended to the weighted case. Denote by $e^{-t\vec{\Delta}}$ the associated heat operator, and by $\vec{p_t}(x,y)$ its kernel. Thus, for every $x$ and $y$ in $M$, $\vec{p_t}(x,y)$ is a linear map from $T^*_yM$ to $T^*_xM$, where $\Lambda^1T^*M$ is the vector bundle of $1$-forms on $M$. By definition, for every compactly supported  smooth $1$-forms $\omega$ and $\eta$,  there holds:

$$\langle e^{-t\vec{\Delta}} \omega,\eta\rangle=\int_M (\vec{p_t}(x,y)\omega(y),\eta(x))_x\,d\mu(x)d\mu(y).$$
We consider the Gaussian estimate for $\vec{p_t}(x,y)$:
\begin{equation}\tag{$\vec{U\!E}$}
\|\vec{p}_{t}(x,y)\|_{y,x}\lesssim
\frac{1}{V(x,\sqrt{t})}\exp
\left(-\frac{d^{2}(x,y)}{Ct}\right), \quad \forall~t>0,\, \mbox{a.e. }x,y\in
 M,\label{vecUE}
\end{equation}
for some $C>0$.
Here $\|\cdot\|_{y,x}$  denotes the norm of  the operator $\vec{p}_{t}(x,y)$ from $T^*_yM$ to $T^*_xM$ endowed with the Riemannian metrics at $y$ and $x$. 

Gaussian estimates for   $p_t$ and $\vec{p_t}$ have important consequences for the Riesz transform: if $M$ satisfies \eqref{d} and \eqref{UE}, then the Riesz transform $d\Delta^{-1/2}$ extends to a bounded operator on $L^p$, for $1<p\le 2$ (see \cite{CD1}), and if in addition it satisfies \eqref{vecUE} then $d\Delta^{-1/2}$ is also bounded for  $2\le p<+\infty$ (see \cite{CD2}, \cite{CD3} and \cite{Sik}). Let us now recall the Bochner formula for the Hodge Laplacian on $1$-forms (for which we need $\mu$ to be the Riemannian measure):
$$\vec{\Delta}=\nabla^*\nabla+\mathrm{Ric}=\bar{\Delta}+\mathrm{Ric},$$
where $\nabla^*\nabla=\bar{\Delta}$ is called the rough Laplacian, and where $\mathrm{Ric}$ is, at every point $x$, a symmetric endomorphism derived in a canonical way from the Ricci curvature tensor. The Bochner formula allows one to see $\vec{\Delta}$ as a ``generalised Schr\"{o}dinger operator'' with potential $\mathrm{Ric}$. If $M$ has non-negative Ricci curvature, \eqref{UE} holds (see \cite{LY}) and \eqref{vecUE} follows  by  Bochner formula and  domination theory (see \cite{Bes}). Hence  the Riesz transform $d\Delta^{-1/2}$ extends to a bounded operator on $L^p$ for all $1<p<+\infty$. The latter result was originally proved by Bakry \cite{Bak}, with probabilistic techniques. However, it is not fully satisfactory, since pointwise lower bounds for the Ricci curvature are easily destroyed by perturbation, whereas the boundedness of the Riesz transform is preserved under ``mild'' perturbation of the metric and the topology (for example, by modifiying the topology or the metric on a compact set, see \cite{CN}, \cite{C}, \cite{D1}). It is thus natural to look for a condition on the Ricci curvature,  allowing some  negativity, that will imply the boundedness of the Riesz transform on all the $L^p$ spaces (for early results in this direction, see \cite{LXD}). As we have just seen, one way to achieve this goal is to obtain Gaussian estimates \eqref{vecUE} for the heat kernel on $1$-forms, under some (small) amount of negative Ricci curvature. In \cite{CZ}, T.  Coulhon and Q. Zhang proved \eqref{vecUE} using domination theory, for manifolds satisfying \eqref{d} and \eqref{UE} and a ``non-collapsing'' assumption for the volume of geodesic balls, under the hypothesis that the negative part of the Ricci curvature is ``small enough'' in an integral sense; more precisely, assuming moreover that the scale-invariant Poincar\'{e} inequalities \eqref{P} hold, the smallness condition on  the negative part of the Ricci curvature $\mathrm{Ric}_-$ can be written as:
$$\sup_{x\in M}\int_M G(x,y)||\mathrm{Ric}_-(y)||_y\,d\mu(y)<\delta,$$
where $\delta$ is some small (non explicit) constant. However, this smallness condition is not satisfactory, because it entails a {\em global} size bound on the negative part of the Ricci curvature, whereas it is expected that a ``smallness at infinity'' (together with other geometric assumptions preventing the manifold from having several ends) is enough. This is supported by the results of \cite{D2}: assuming that $M$ satisfies a global Sobolev inequality of dimension $n>4$, and has Euclidean volume growth, with $\mathrm{Ric}_-$ in $L^{\frac{n}{2}\pm\varepsilon}$, it is proved that \eqref{vecUE} holds if and only if there is no non-zero $L^2$ harmonic $1$-form. Let us mention at this point that exact $L^2$ harmonic $1$-forms are related to the number of ends of the manifold; in particular, if the Sobolev inequality holds, then the hypothesis that there is no non-zero $L^2$ harmonic $1$-form implies that $M$ has only one end, see \cite{LT}. The condition on $\mathrm{Ric}_-$ is more satisfactory here, because there is no size limitation, more precisely $||\mathrm{Ric}_-||_{\frac{n}{2}\pm\varepsilon}$ can be as large as we want.
 However, the assumption that $M$ satisfies a global Sobolev inequality is not natural, and it is desirable to have a result for manifolds that merely satisfy \eqref{d} and \eqref{UE}. In this case, it is not obvious which smallness condition on the negative part of the Ricci curvature is analogous to $L^{\frac{n}{2}\pm\varepsilon}$. Also, the proof in \cite{D2} uses in several crucial ways the global Sobolev inequality, and thus it is not easy to adapt it to the broader setting of manifolds satisfying only  \eqref{d} and \eqref{UE}. In this article, we extend   the results of \cite{D2} to this more general setting, by mixing the ideas from \cite{D2} with the techniques from \cite{BCS}.

\bigskip

We shall work in the framework of generalised Schr\"{o}dinger operators acting on sections of a vector bundle over $M$. More precisely, consider a generalised Schr\"{o}dinger operator
$$\mathcal{L}=\nabla^*\nabla + \mathcal{R},$$
acting on a finite-dimensional Riemannian  bundle $E\rightarrow M$,  that is a finite-dimensional vector bundle equipped with a scalar product $\left(\cdot, \cdot\right)_x$ depending continuously on $x\in M$ (see for instance \cite[Section E]{Be}). Here $\nabla$ is a connection on $E\rightarrow M$ which is  compatible with the metric, and $\nabla^*\nabla$ is the so-called ``rough Laplacian'' which we will from now on  denote  by $\bar{\Delta}$. Of course, the formal adjoint $\nabla^*$ depends on the measure $\mu$, and we really should write $\nabla^*_\mu$ instead of $\nabla^*$, but to keep notations light we prefer to keep this dependance implicit. The ``potential'' $\mathcal{R}$ is by definition a $L^\infty_{loc}$ section of the vector bundle $\mathrm{End}(E)$, that is, for all $x\in M$, $\mathcal{R}(x)$ is a symmetric endomorphism of $E_x$, the fiber at $x$. Notice that if $E$ is the trivial $M\times \R$, a generalised Schr\"{o}dinger operator on $E$ is just a {\em scalar} Schr\"{o}dinger operator $\Delta+V$, where $V: M\rightarrow \R$ is a real potential. Since $\mathcal{R}$ is in $L^\infty_{loc}$, by standard elliptic regularity, solutions of $\mathcal{L}\omega=0$ are contained in $C_{loc}^{1,\alpha}$ for all $\alpha\in (0,1)$. For a.e. $x\in M$, one can diagonalize $\mathcal{R}(x)$  in an orthonormal basis of $E_x$.
Denote by $\mathcal{R}_+(x)$ the endomorphism corresponding to the non-negative eigenvalues, and  by $-\mathcal{R}_-(x)$ the one corresponding to the negative eigenvalues, so that $\mathcal{R}_+(x)$, $\mathcal{R}_-(x)$ are a.e. non-negative symmetric endomorphisms acting on the fiber $E_x$ and

$$\mathcal{R}=\mathcal{R}_+-\mathcal{R}_-.$$
Notice also that $\mathcal{R}_+$ and $\mathcal{R}_-$ belong to $L^\infty_{loc}$. Denote by $|\cdot|_x$ the norm on $E_x$ derived from $\left(\cdot, \cdot\right)_x$ and by $\|\cdot\|_x$ the induced norm on    $\mathrm{End}(E_x)$. In particular,
$$\|\mathcal{R}_-(x)\|_x=\sup_{v\in E_x,|v|_x=1}|\left(\mathcal{R}_-(x)v,v\right)_x|=\max \sigma(\mathcal{R}_-(x)),$$
where $\sigma(\mathcal{R}_-(x))$ is the (finite)  set of eigenvalues of $\mathcal{R}_-(x)$. Let  $C^\infty(E)$ (resp. $C_0^\infty(E)$) be the set of smooth sections of $E$ (resp. of smooth `compactly supported sections'', that is sections that coincide with the zero section outside a compact set). For $p\geq 1$ we will consider the $L^p$-norm on sections of $E$:
$$||\omega||_p=\left(\int_M|\omega(x)|_x^p\,d\mu(x)\right)^{1/p}$$
with the usual extension for $p=\infty$.
We shall denote  by $L^p(E)$, or simply $L^p$ when  no confusion is possible,  the set of sections of $E$ with finite $L^p$-norm, modulo equality a.e.. Sometimes, depending on the context, $L^p$ will simply refer to real-valued functions. We will denote by 
$$\langle\omega_1,\omega_2\rangle:=\int_M(\omega_1(x),\omega_2(x))_x\,d\mu(x)$$ the scalar product in $L^2(E)$.

From an obvious adaptation of Strichartz's proof that the Laplacian is self-adjoint on a complete manifold (see Theorem 3.13 in \cite{PRS}), we know that if $\mathcal{R}_-$ is bounded, then $\mathcal{L}=\bar{\Delta}+\mathcal{R}_+-\mathcal{R}_-$ is essentially self-adjoint on $C_0^\infty(E)$. If $\mathcal{R}_-$ is not bounded, then we will consider the Friedrichs extension of $\mathcal{L}$, which is an unbounded, self-adjoint operator on $L^2(E)$. We will make a crucial assumption on $\mathcal{L}$: we will assume that it is {\em non-negative}. By this we mean that the associated quadratic form to $\mathcal{L}$ is non-negative, that is for every $\omega\in C_0^\infty(E)$,

$$Q_\mathcal{L}:=\int_M|\nabla \omega|^2+\langle\mathcal{R}\omega,\omega\rangle\geq0.$$
By the spectral theorem, one can then consider $e^{-t\mathcal{L}}$, which is a contraction semigroup on $L^2$. By elliptic estimates again, there exists a kernel $p_t^{\mathcal{L}}$ for $e^{-t\mathcal{L}}$, that is, for every $x,y\in M$, $p_t^{\mathcal{L}}(x,y)$ is a linear map from $E_y$ to $E_x$ such that for every $\omega_i\in L^2(E)$, $i=1,2$,

$$\langle e^{-t\mathcal{L}}\omega_1,\omega_2\rangle=\int_{M\times M}\left(p_t^{\mathcal{L}}(x,y)\omega_1(y),\omega_2(x)\right)_x\,d\mu(x)d\mu(y).$$
By self-adjointness of $\mathcal{L}$, the adjoint of $p_t^{\mathcal{L}}(x,y)$ is $p_t^{\mathcal{L}}(y,x)$ a.e.. Denote by $\|\cdot\|_{y,x}$ the operator norm on $\mathrm{End}(E_y,E_x)$. We will say that the heat kernel associated with $\mathcal{L}$ satisfies Gaussian estimates if for every $t>0$ and a.e $x,y\in M$,

\begin{equation}\label{UEL}\tag{$U\!E_{\mathcal{L}}$}
||p_t^{\mathcal{L}}(x,y)||_{y,x}\lesssim
\frac{1}{V(x,\sqrt{t})}\exp
\left(-\frac{d^{2}(x,y)}{Ct}\right).
\end{equation}
We also consider the on-diagonal estimates for the heat kernel of $\mathcal{L}$: for every $t>0$ and a.e. $x\in M$,

\begin{equation}\label{DUEL}\tag{$DU\!E_{\mathcal{L}}$}
||p_t^{\mathcal{L}}(x,x)||_{x,x}\lesssim
\frac{1}{V(x,\sqrt{t})}.
\end{equation}
Thanks to a general result of A. Sikora (\cite[Theorems 1 and 4]{Sik}), if $\mathcal{L}$ satisfies the Davies-Gaffney estimates, then \eqref{UEL} and \eqref{DUEL} are {\em equivalent}. It turns out that in our setting, $\mathcal{L}$ does indeed satisfy the Davies-Gaffney estimates. This is explained in more details in the Appendix (see Theorem \ref{off-diago} therein). Thus, again in our setting, \eqref{UEL} is equivalent to \eqref{DUEL}.

Let us denote by $W(x)$ the greatest eigenvalue of $\mathcal{R}_-(x)$. Classical domination theory (see \cite{Bes}) implies that,  in the {\em unweighted case}, that is if $\mu$ is the Riemannian measure, for every $\omega\in C_0^\infty(E)$,

\begin{equation}\label{Dom1}
|e^{-t\mathcal{L}}\omega|\leq e^{-t(\Delta-W)}|\omega|.
\end{equation}
It is an important fact that this remains true in the {\em weighted case} as well, see the Appendix for details about this claim. In particular, if $\mathcal{R}$ is everywhere non-negative, for every $\omega\in C_0^\infty(E)$,

\begin{equation}\label{Dom2}
|e^{-t\mathcal{L}}\omega|\leq e^{-t\Delta}|\omega|.
\end{equation}
If furthermore $M$ satisfies \eqref{UE}, one then deduces from \eqref{Dom2} that \eqref{UEL} is satisfied.\\

As we  mentioned before for the case of $1$-forms, an important example of such a generalised Schr\"{o}dinger operator $\mathcal{L}$ is provided in the unweighted case by $\Delta_k=d^*_{k}d_k+d_{k-1}d^*_{k-1}$, the Hodge Laplacian acting on differential $k$-forms (of course, $\Delta_1$ is nothing but $\vec{\Delta}$). It obviously follows  from its definition that $\Delta_k$ is non-negative.  Let us denote by $\Lambda^kT^*M$ the vector bundle of smooth $k$-forms. As shown in \cite{GM}, the operator $\Delta_k$ admits a \textit{Bochner decomposition}:
$$\Delta_k=\bar{\Delta}+\mathscr{R}_k.$$
Here $\mathscr{R}_k$ is a symmetric section of $\mathrm{End}(\Lambda^kT^*M)$ derived from the Riemann curvature operator. An important particular case is $k=1$, for which $\mathcal{R}_1$ is canonically identified with the Ricci curvature. We may denote in short the Gaussian estimate $(UE_{\Delta_k})$ 
by $(\vec{U\!E}_k)$. Note that $(\vec{U\!E}_1)$ is nothing but $(\vec{U\!E})$ already introduced. It turns out that if $\mu$ is not the Riemannian measure, the corresponding {\em weighted} Hodge Laplacians are also of interest: such operators have been considered by E. Witten \cite{Wit} and J-M. Bismut \cite{Bis}, in order to give a new proof of the Morse inequalities, and later by E. Bueler \cite{Bue} to study the cohomology of non-compact manifolds. 
Explicitely, the weighted Hodge Laplacian acting on $k$-forms is defined by
$$\Delta_{\mu,k}=dd_{\mu}^*+d^*_{\mu} d$$ 
(in order to keep notation reasonably light, we make the dependance on $k$ implicit in $d$ and $d^*$; also, we shall often denote $\Delta_{\mu,1}$ by $\vec{\Delta}_{\mu}$), where $d^*_{\mu}$ is the adjoint of $d$, which thus depends on the measure $\mu$. It turns out that a Bochner-type formula still holds for the weighted Hodge-Laplacian; this is discussed in more details in the Appendix. For the time being, let us write without justification the weighted Bochner formula for the Hodge Laplacian:

\begin{equation}\label{Boch_w}
\Delta_{\mu,k}\,\omega=\bar{\Delta}_\mu\,\omega+\mathscr{R}_k\,\omega-\mathscr{H}_f\,\omega,
\end{equation}
where $\bar{\Delta}_\mu=\nabla^*_\mu\nabla$ is the weighted rough Laplacian, and the additional term $\mathscr{H}_f$ is the ``Hessian operator'' of $f$ (see the Appendix), which is a symmetric section of the vector bundle $\mathrm{End}(\Lambda^kT^*M)$. Thus, the term $\mathscr{R}_k-\mathscr{H}_f$ is indeed a generalised potential in the sense of our previous definition. In the particular case $k=1$, the ``potential'' term in formula \eqref{Boch_w} is $\mathrm{Ric}_\mu:=\mathrm{Ric}-\mathscr{H}_f$,  the same term that appears in the iterated {\em carr\'{e} du champ} $\Gamma_2$ of a weighted Laplacian on a Riemannian manifold (see \cite[Proposition 3]{BE}). This is not a coincidence, because the (Bochner) formula for the {\em carr\'{e} du champ} of a weighted Laplacian actually follows from  formula \eqref{Boch_w} for  $1$-forms (this is explained in greater details in the Appendix).

%
%

\bigskip

Coming back to the case of a generalised Schr\"{o}dinger operator, consider the following {\em strong subcriticality} property, a strengthening of  non-negativity which may or may not be satisfied by $\mathcal{L}$: there exists $0<\eta < 1$ such that for every $\omega\in C_0^\infty(E)$,

\begin{equation}\label{SP}
\langle \mathcal{R}_-\omega,\omega\rangle\leq (1-\eta)\langle (\bar{\Delta}+\mathcal{R}_+) \omega,\omega\rangle.
\end{equation}
If \eqref{SP} is satisfied, we will say that $\mathcal{L}$ is $(1-\eta)$-strongly subcritical. Notice that if $M$ is non-parabolic then \eqref{SP} implies $\mathrm{Ker}_{L^2}(\mathcal{L})=0$ (see Lemmas \ref{L7} and \ref{L4c}). This strong subcriticality (or strong positivity) assumption first appeared in the work of B. Simon and E. B. Davies (see \cite{DS}). It has then been used in numerous papers dealing with  perturbation theory of Schr\"{o}dinger operators (see \cite{T}, \cite{AO} and the references therein). In relation to  Riesz transforms and heat kernels on one-forms, this strong subcriticality condition is used in \cite{AO, CZ, M}. In \cite{CZ}, a pointwise estimate weaker than \eqref{vecUE} is proved (basically,  \eqref{vecUE} multiplied by a polynomial factor in time) for manifolds satisfying \eqref{d} and \eqref{UE},  non-collapsing of the volume of balls of radius $1$,  such that $\mathrm{Ric}_-\in L^p\cap L^\infty$ for some $1\leq p<+\infty$ and $\vec{\Delta}$ is strongly subcritical. In \cite{M}, some $L^p\to L^q$ off-diagonal estimates for $e^{-t\vec{\Delta}}$ are proved under \eqref{d}, \eqref{UE}, and strong subcriticality of $\mathrm{Ric}_-$. The exponents $p$ and $q$ that are allowed depend on the best constant in the strong subcriticality inequality \eqref{SP} for $\mathcal{R}=\mathrm{Ric}$. \\

We now come to our ``smallness at infinity'' assumption: we shall say that condition \eqref{condik} is satisfied by a  section $\mathcal{V}\in L^\infty_{loc}$ of the vector bundle $\mathrm{End}(E)$ on $M$ if  $M$ is non-parabolic and  there is a compact subset $K_0$ of $M$ such that
\begin{equation}\label{condik}\tag{$K$}
\sup_{x\in M}\int_{M\setminus K_0}G(x,y)\|\mathcal{V}(y)\|_y\,d\mu(y)<1,\end{equation}
where $G$ is the Green function for the Laplace operator on functions.
In a more compact way, condition \eqref{condik}  may be formulated as $\|\Delta^{-1}(\|\mathcal{V}\|\mathbf{1}_{M\setminus K_0})\|_{\infty}<1$.
In the sequel, we shall identify a section of the  vector bundle $\mathrm{End}(E)$ with the associated composition operator. 
Under this identification, the quantity $\|\Delta^{-1}(\|\mathcal{V}\|\mathbf{1}_{M\setminus K_0})\|_{\infty}$ is also the $L^p-L^p$ norm of the composition operator by
$\Delta^{-1}(\|\mathcal{V}\|\mathbf{1}_{M\setminus K_0})$ on $E$ for all $p\in[1,+\infty]$. Let us mention that \eqref{condik} was introduced in \cite{D3} to study the heat kernel and the Riesz transform of Schr\"{o}dinger operators with potentials that are ``small at infinity''. In fact, \eqref{condik} is a generalization of the more familiar Kato class at infinity $K^\infty(M)$, defined by: $\mathcal{V}\in K^\infty(M)$ if

\begin{equation}\label{Kato}
\lim_{R\to\infty}\sup_{x\in M}\int_{M\setminus B(x_0,R)}G(x,y)\|\mathcal{V}(y)\|_y\,d\mu(y)=0,
\end{equation}
for some (all) $x_0\in M$. This is a slightly more restrictive condition than our condition \eqref{condik}. The notion of Kato class at infinity is often used as an effective tool to obtain Gaussian estimates for the heat kernel of Schr\"{o}dinger operators (see \cite{T}). Also, in \cite{D3}, a sufficient condition for a potential to satisfy \eqref{Kato} is presented, in terms of some weighted $L^p$ spaces: if \eqref{UE} holds, and there is $\varepsilon>0$ such that 

$$\mathcal{V}\in L^{\frac{\nu}{2}-\varepsilon}\left(M,\frac{d\mu(x)}{V(x,1)}\right)\cap L^{\frac{\nu'}{2}+\varepsilon}\left(M,\frac{d\mu(x)}{V(x,1)}\right),$$
where $\nu$ and $\nu'$ are the exponents from \eqref{d} and \eqref{rnu},
then $\mathcal{V}$ belongs to the Kato class at infinity $K^\infty(M)$, and thus in particular satisfies \eqref{condik}. In the case where the volume growth is polynomial with exponent  $n>2$, that is
$$C^{-1}r^n\leq V(x,r)\leq Cr^n,\qquad \forall x\in M,\,\forall r>0,$$ the above condition on $\mathcal{V}$ is the familiar condition $\mathcal{V}\in L^{\frac{n}{2}\pm \varepsilon}$.

\subsection{Our results}

We can now state our main result, which generalizes \cite[Theorem 4, Proposition 11]{D2} to manifolds that do not necessarily satisfy a global Sobolev inequality:

\begin{Thm}\label{main}
Let $(M,\mu)$ be  a complete, non-compact, connected, weighted Riemannian manifold satisfying  \eqref{d} and \eqref{UE}. 
Assume that $(M,\mu)$ satisfies the volume lower bound \eqref{dd} with  $\kappa>2$.
Let $E$ be a vector bundle with  basis $M$ and a connection $\nabla$ compatible with the metric, and let 
$$\mathcal{L}=\nabla^*_\mu\nabla+\mathcal{R}_+-\mathcal{R}_-$$
be a generalised (weighted) Schr\"{o}dinger operator on $E$, such that $\mathcal{R}_-$ satisfies condition~$\eqref{condik}$. Then the following are equivalent:

\begin{enumerate}

\item[i)]  \eqref{UEL} holds.

\item[ii)]  $\mathcal{L}$ is strongly subcritical.

\item[iii)]  $\mathrm{Ker}_{L^p}(\mathcal{L})=\{0\}$, for some (any) $p\in \left(\frac{\kappa}{\kappa-2},+\infty\right)$. 

\end{enumerate}
If in addition
 $\kappa>4$, then $\rm iii)$ is equivalent to $\mathrm{Ker}_{L^2}(\mathcal{L})=\{0\}$.

\end{Thm}
We now make a number of comments about our result:

\begin{Rem}
{\em
Theorem \ref{main} greatly improves upon \cite[Theorem 3]{D2}: first the geometric assumptions on $M$ are considerably weaker (no global Sobolev inequality assumed, a weaker volume growth condition). Secondly, in \cite{D2} it wasn't noticed that the strong subcriticality of $\mathcal{L}$ not only implies the Gaussian estimates for $e^{-t\mathcal{L}}$, but is actually {\em equivalent} to it. We also would like to point out that there is a misprint in the statement of \cite[Theorem 3]{D2}: $n$ should be strictly larger than $4$, as follows from \cite[Proposition 11]{D2}.
}
\end{Rem}

\begin{Rem}\label{first}
{\em A sufficient condition for $\mathcal{L}$ to be strongly subcritical is that
 
$$ \sup_{x\in M}\int_{M}G(x,y)\|\mathcal{R}_-(y)\|_y\,d\mu(y)<1.$$
In other words, if one can take $K_0=\emptyset$ in \eqref{condik}, then $\mathcal{L}$ is strongly subcritical hence, under the assumptions of  Theorem \ref{main},  \eqref{UEL} is satisfied. In the case where $M$ satisfies \eqref{d} and \eqref{due}, the above implication follows from  Proposition \ref{goh} below (see Section \ref{harm} for details). By using \cite[Lemma 3.5]{D3}, one can in fact show the same statement for every non-parabolic manifold.

}
\end{Rem}
\begin{Rem}\label{second}
{\em 
It is not hard to see that for every $p\in [1,\infty)$, $\mathrm{Ker}_{L^p}(\mathcal{L})=\{0\}$ is a necessary condition for \eqref{UEL} to hold (see Lemma \ref{converse}). Our result asserts that if $p>\frac{\kappa}{\kappa-2}$, then under our other assumptions it is also sufficient.
}
\end{Rem}

\begin{Rem}\label{subc}
{\em 

In Theorem \ref{main}, if instead of \eqref{dd} for some $\kappa>2$, one only assumes that $M$ is non-parabolic, then it is still true that $\mathcal{L}$ strongly subcritical implies the validity of  \eqref{UEL}. See Section 5 for details.
}
\end{Rem}
In the case of a {\em scalar} Schr\"{o}dinger operator, we can further strengthen  Theorem \ref{main}.
Before stating this result, recall that a non-negative, {\em scalar} Schr\"{o}dinger operator is called subcritical if it has positive minimal Green functions (see e.g. \cite{PT}). It is well-known that if $M$ is non-parabolic, then strong subcriticality implies subcriticality: this follows from  one of the many equivalent characterizations of subcriticality. In fact, a Schr\"{o}dinger operator $L$ is subcritical if and only if, for one (or every) $0\leq W\in C_0^\infty(M)$, $W\neq 0$, $L-\varepsilon W\geq0$ for $\varepsilon>0$ small enough. Note that (1.4) implies that $Q_L(u)\geq \varepsilon \int_M |\nabla u|^2$. Combining this with the fact (see for instance \cite[pp.46--47]{A} or \cite{PT}) that $M$ is non-parabolic if and only if there exists a positive function $\rho$ such that
$$\int_M \rho u^2\leq \int_M|\nabla u|^2,\qquad\forall u\in C_0^\infty(M),$$
one obtains easily that strong subcriticality implies subcriticality for $M$ non-parabolic. Then, in the scalar case, one can get rid of the condition \eqref{dd} in Theorem \ref{main}, if one is ready to replace the characterization of Gaussian estimates in terms of non-existence of $L^p$ $L$-harmonic functions by the subcriticality of the operator:

\begin{Cor}\label{scalar}

Let $(M,\mu)$ be a non-parabolic, complete, non-compact, connected, weighted  Riemannian manifold satisfying  \eqref{d} and \eqref{UE}. Let $L=\Delta_\mu+V$ be a non-negative, {\em scalar} weighted Schr\"{o}dinger operator on $M$, and assume that $V_-$ satisfies \eqref{condik}. Then the following are equivalent:

\begin{enumerate}

\item[i)]  $(U\!E_{L})$ holds. 

\item[ii)] $L$ is strongly subcritical.

\item[iii)] $L$ is subcritical.

\end{enumerate}

\end{Cor}
\begin{Rem}
{\em

Corollary \ref{scalar} improves \cite[Theorem 4.1]{D3}, because no assumption is made on $V_+$; actually, it answers a question asked in this paper. Also, it should be emphasized that Theorem~\ref{main} does \textit{not} follow   from Corollary \ref{scalar} by domination theory. Indeed, domination theory yields
$$|e^{-t\mathcal{L}}\omega|\leq e^{-t(\Delta-W)}|\omega|,\qquad\forall \omega\in C_0^\infty(E),$$ 
where $W(x)$ is the greatest eigenvalue of $\mathcal{R}_-(x)$; but in general, the non-negativity of $\mathcal{L}$ does not imply the non-negativity of  $\Delta-W$. Therefore, in general, Corollary \ref{scalar} cannot be applied to $\Delta-W$.

}

\end{Rem}


\begin{Rem}\label{optimal1}
{\em
For $M=\R^n\sharp\R^n$, the connected sum of two copies of $\R^n$, it follows from \cite{LT} that if $n\geq 3$ then there exists a non-constant harmonic function $h$  with $dh\in L^2$. Thus, $dh\in\mathrm{Ker}_{L^2}(\D)$, and according to Lemma \ref{converse} below \eqref{vecUE} cannot hold.  For the case $n=2$ see  Remark \ref{two-planes}.

}
\end{Rem}

\begin{Rem}\label{optimal2}
{\em

The interval $\left(\frac{\kappa}{\kappa-2},+\infty\right)$ in Theorem \ref{main}, iii) is {\em optimal}, even in the scalar case, as the following example demonstrates: consider a function $u\in C^\infty(\R^3)$ such that $u(x)=|x|^{-1}$ for all $|x|\geq 1$,
$u(x) > 0$ for all $x\in \R^3$. 
Then the potential  $V=\frac{\Delta u}{u}$ is smooth and compactly supported. Let $M=\R^3$ and $L=\Delta-V$. Notice that  $\kappa=3$ hence $\frac{\kappa}{\kappa-2}=3$. Moreover  $L$ is non-negative, $(U\!E_{L})$ does not hold, and 
$$\mathrm{Ker}_{L^p}(L)=\left\{\begin{array}{lcr}
\{0\},\,\mbox{for }1\leq p\leq 3,\\
\neq \{0\}\,\mbox{for }p>3.
\end{array}\right.$$ 
This is proved in Proposition \ref{exa} below.
}
\end{Rem}

In order to even define condition \eqref{condik}, one needs $M$ to be non-parabolic. In the case where $M$ is parabolic, it is also relevant to ask for a condition on $\mathcal{R}_-$ that  ensures the validity of \eqref{UEL}. In this direction, we have the following partial result.

\begin{Thm}\label{thm_para}

Assume that $M$ satisfies \eqref{d} and \eqref{UE}, and that $\mathcal{L}$ is strongly subcritical. Then

\begin{enumerate}

\item[i)] If $\nu<2$, then \eqref{UEL} holds.

\item[ii)] If $\nu=2$, then $(V\!EV_{p,q,\gamma})$ (see \cite{BCS} or Section $\ref{resemi}$ below) holds for all $1<p\leq q<+\infty$ and $\gamma\in \R$. In particular, $e^{-t\mathcal{L}}$ is uniformly bounded on $L^p$, for all $p\in (1,+\infty)$.

\end{enumerate}

\end{Thm}
This shows that when $\nu<2$, no smallness at infinity condition on $\mathcal{R}_-$ is needed. Theorem \ref{thm_para} improves upon the case $\kappa<2$ in \cite[Theorem 4.1.15]{BCS}, where a stronger strong subcriticality assumption was made. Part i) also improves a result of Magniez \cite{M}, who proved that under \eqref{d}, \eqref{UE}, if $\mathcal{L}$ is strongly subcritical and $\nu\leq 2$, then the semi-group $e^{-t\mathcal{L}}$ is uniformly bounded on $L^p$, for all $p\in (1,+\infty)$. Statement ii) can also be deduced from  results of Magniez \cite{M} and Assaad-Ouhabaz \cite{AO}, but we provide a different, shorter proof. We do not know whether or not the full Gaussian estimates \eqref{UEL} hold in the case $\nu=2$ (but we suspect that they do not, without additional assumptions on $\mathcal{R}_-$). Let us mention also that it follows from the techniques in \cite{AO} and \cite{M}, together with results of \cite{BCS}, that if $\nu>2$ and $\mathcal{L}$ is $(1-\eta)$-strongly subcritical, then $(V\!EV_{p,q,\gamma})$ holds for all $p_0'< p\leq q<p_0$, $p_0=\frac{2\nu}{(\nu-2)(1-\sqrt{1-\eta})}$, and $\gamma\in \R$. Compare with the case $\kappa\geq 2$ in \cite[Theorem 4.1.15]{BCS}.

\begin{Rem}\label{two-planes}

{\em

When $\nu=\kappa=2$, one can ask what is the relationship between \eqref{vecUE}, the strong subcriticality of $\mathcal{L}$, and the condition $\mathrm{Ker}_{L^p}(\vec{\Delta})=\{0\}$ for some values of $p$. Notice that when $\kappa=2$, the condition $\mathrm{Ker}_{L^p}(\vec{\Delta})=\{0\}$,  $p\in\left(\frac{\kappa}{\kappa-2},\infty\right)$ becomes empty. Thus, the exponent $\frac{\kappa}{\kappa-2}$ may have to be replaced by another exponent. The following example provides a useful insight in this respect: take $M=\R^2\sharp\R^2$, the connected sum of two Euclidean planes. Then, on $M$ the Laplacian on $1$-forms $\vec{\Delta}$ is {\em not} strongly subcritical, and does not satisfy \eqref{vecUE}. Also, 

$$\mathrm{Ker}_{L^p}(\vec{\Delta})\neq\{0\},\,p>2,$$
and if moreover the gluing of the two planes is done in such a way that $M$ has genus $0$, then

$$\mathrm{Ker}_{L^2}(\vec{\Delta})=\{0\}.$$
See Proposition \ref{Lpara} for a proof of these facts and compare with Remark \ref{optimal2}. Given this example, one can ask the following question: does the strong subcriticality of $\mathcal{L}$ together with $\mathrm{Ker}_{L^p}(\mathcal{L})=\{0\}$, $p\in (2,\infty)$ imply \eqref{UEL} on a manifold with quadratic volume growth (such as $\R^2$)?

}
\end{Rem}
It is worth stating Theorem \ref{main} in the particular case where $\mathcal{L}=\Delta_k$, the Hodge Laplacian acting on $k$-forms:

\begin{Cor}\label{hodge}
Let $M$ be a complete non-compact connected manifold satisfying  \eqref{d} and \eqref{due}, such that the negative part of the curvature operator $\mathscr{R}_k$ satisfies  condition $\eqref{condik}$.  Assume also  \eqref{dd} for some $\kappa>2$. Then the heat kernel of $\Delta_k$ satisfies $(\vec{U\!E}_k)$ if and only if $\mathrm{Ker}_{L^p}(\Delta_k)=\{0\}$, for some (any) $p\in \left(\frac{\kappa}{\kappa-2},\infty\right)$. In particular, if $\kappa>4$, the heat kernel of $\Delta_k$ satisfies $(\vec{U\!E}_k)$ if and only if $\mathrm{Ker}_{L^2}(\Delta_k)=\{0\}$. 
\end{Cor}
The result of Corollary \ref{hodge} for the Hodge Laplacian acting on $1$-forms has important consequences in terms of further estimates for $p_t$ and  $L^p$-boundedness of the Riesz transform. Indeed, it has been observed in \cite[Section 5.2]{CD3} that $(\vec{U\!E})$ together with \eqref{d} and \eqref{UE}  implies  the following Gaussian estimate for the gradient of the heat kernel:
\begin{equation}\label{gradient}\tag{$G$}
|\nabla_xp_t(x,y)|\lesssim \frac{1}{\sqrt{t}V(x,\sqrt{t})}\exp\left(-\frac{d^{2}(x,y)}{ct}\right),\, \forall~t>0, \,\mbox{a.e. }x,y\in
 M.
\end{equation}
It is well-known (see for instance \cite{LY}) that \eqref{gradient} together with \eqref{d} and \eqref{due} implies the Gaussian heat kernel lower bound. It is also known (see \cite{ACDH}) that \eqref{gradient}  has in turn important consequences for the boundedness of the Riesz transform on $L^p$ (in our case where a Gaussian estimate on $1$-forms is available, the boundedness of the Riesz transform follows also more easily from \cite{Sik}). Let us introduce, for $1<p<+\infty$, the inequality

\begin{equation}\label{R_p}\tag{$R_p$}
\| |\nabla u| \|_p\lesssim |\|\Delta^{1/2}u\|_p,\quad\forall u\in C_0^\infty(M),
\end{equation}
which is equivalent to the Riesz transform $d\Delta^{-1/2}$ being bounded on $L^p$. We also introduce the reverse inequality:

\begin{equation}\label{RR_p}\tag{$RR_p$}
\|\Delta^{1/2}u\|_p\lesssim \| |\nabla u| \|_p,\quad\forall u\in C_0^\infty(M),
\end{equation}
as well as the combination of \eqref{R_p} and \eqref{RR_p}:

\begin{equation}\label{E_p}\tag{$E_p$}
\||\nabla u|\|_p\simeq ||\Delta^{1/2}u||_p,\quad\forall u\in C_0^\infty(M).
\end{equation}
The main result of  \cite{ACDH} is that \eqref{gradient} together with \eqref{d} and \eqref{due} (in fact it has been shown later in \cite{CS2} that  \eqref{due} is implied by the other two assumptions) implies \eqref{E_p} for all $p\in (1,\infty)$, hence we have the following consequence of Corollary  \ref{hodge}:

\begin{Cor}\label{4cons}
Let $(M,\mu)$ be a weighted, complete non-compact connected manifold satisfying  \eqref{d} and \eqref{UE}. Recall the weighted Ricci tensor $\mathrm{Ric}_\mu:=\mathrm{Ric}-\mathscr{H}_f$, if $\mu=e^f\nu$. Assume that $\left(\mathrm{Ric}_\mu\right)_-$ satisfies condition \eqref{condik}.  Assume moreover that $M$ satisfies \eqref{dd} for some $\kappa>2$, and that $\mathrm{Ker}_{L^p}(\vec{\Delta}_{\mu})=\{0\}$ for some $p\in (\frac{\kappa}{\kappa-2},\infty)$. Then:

\begin{enumerate}

\item The  heat kernel on functions also satisfies Gaussian upper and lower bounds, i.e. \eqref{LY} holds; equivalently, the scale-invariant Poincar\'{e} inequalities \eqref{P} hold.

\item The Gaussian estimate for the gradient of the heat kernel \eqref{gradient}  holds.

\item $(E_p)$ holds  for all $p\in (1,\infty)$.

\end{enumerate}

\end{Cor}

\begin{Rem}
{\em
More generally, the conclusions of Corollary \ref{4cons} hold if one only assumes that $M$ is non-parabolic, satisfies \eqref{d}, \eqref{UE}, and $\left(\mathrm{Ric}_\mu\right)_-$ satisfies \eqref{condik} and $\vec{\Delta}_{\mu}$ is strongly subcritical (see Remark \ref{subc}). It also holds if $M$ satisfies \eqref{d}, \eqref{UE}, $\nu<2$ and $\vec{\Delta}_{\mu}$ is strongly subcritical (see Theorem \ref{thm_para}).
}
\end{Rem}
We can also deduce a boundedness result for the Riesz transform on $k$-forms; first let us introduce  some elements of Hodge theory on non-compact manifolds. Define $H_2^k(M)$, the $k$-th space of reduced $L^2$ cohomology, by

$$
H_2^k(M)=\frac{\{\omega\in L^2(\Lambda^kT^*M)\,;\,d\omega=0\}}{\mathrm{cl}(dC_0^\infty(\Lambda^{k-1}T^*M))},
$$
where $\mathrm{cl}$ denotes the closure for the $L^2$ norm. Define $\mathcal{H}^k(M)$ to be the space of $L^2$ harmonic $k$-forms, that is the space of $\omega\in L^2(\Lambda^kT^*M)$ such that $\Delta_k\omega=0$ (a priori in the distribution sense, and a posteriori in the $C^\infty$ sense by elliptic regularity). It is known that

$$H_2^k(M)\simeq \mathcal{H}^k(M)=\{\omega\in L^2(\Lambda^kT^*M)\,;\,d\omega=d^*\omega=0\},$$
where $d\omega$ and $d^*\omega$ are meant in the distribution sense (see \cite{DR}).

\begin{Cor}\label{Riesz_hodge}

Let $M$ be a complete non-compact connected manifold satisfying  \eqref{d}, \eqref{UE}, and \eqref{dd} for some $\kappa>4$. Let $\mathscr{R}_k$ be the curvature operator  so that the Bochner formula for the Hodge Laplacian on $k$-forms writes $\Delta_k=\bar{\Delta}+\mathscr{R}_k$. Assume that $(\mathscr{R}_k)_-$ satisfies condition \eqref{condik}, and that $H_2^k(M)=\{0\}$. Then the Riesz transform on $k$-forms $(d_k+d_{k-1}^*)\Delta_k^{-1/2}$ is bounded on $L^p$, for all $p\in (1,2]$. If moreover $(\mathscr{R}_{k+1})_-$ (resp. $(\mathscr{R}_{k-1})_-$) satisfies \eqref{condik} and  $H^{k+1}_2(M)=\{0\}$ (resp. $H^{k-1}_2(M)=\{0\}$), then $d_k\Delta_k^{-1/2}$ (resp. $d_{k-1}^*\Delta_k^{-1/2}$) is bounded on $L^p$, for all $p\in (1,\infty)$.

\end{Cor}
Of course, there is a corresponding statement in the case $2<\kappa\leq 4$ where the condition $\mathrm{Ker}_{L^p}(\Delta_k)=\{0\}$ appears instead of $H_2^k(M)=\{0\}$. There is also a weighted version of Corollary \ref{Riesz_hodge}. We leave it to the reader.

The result of Corollary \ref{Riesz_hodge} is useful because one knows how to compute the reduced $L^2$ cohomology of $M$ in several important cases, for example if $M$ is asymptotically Euclidean, or flat outside a compact set (see \cite{C2} and references therein). However, let us point out that the conditions $H_2^{k-1}(M)=H_2^k(M)=H_2^{k+1}(M)=\{0\}$ are not optimal to obtain the boundedness of the Riesz transform on $k$-forms: for example, for $k=0$ Corollary \ref{Riesz_hodge} states that the Riesz transform is bounded on all the $L^p$ spaces, on an asymptotically Euclidean manifold (of dimension $>4$) if $H^1_2(M)=\{0\}$. By results of \cite{C3}, one knows how to compute $H^1_2(M)$. Indeed, the asymptotically Euclidean manifold $M$ is diffeomorphic to a compact manifolds with $b$ punctures, i.e. $M\simeq\bar{M}\setminus \{p_1,\cdots,p_b\}$, and each $p_i$ corresponds to an end of $M$; then

$$\mathrm{dim}(H_2^1(M))=b_1(\bar{M})+b-1,$$
so that $H_2^1(M)=\{0\}$ if and only if $M$ has only one end, and $b_1(\bar{M})$, the first Betti number of the compact manifold $M$, is zero. However, it follows from  \cite{CCH} that for an asymptotically Euclidean manifold, the boundedness on all $L^p$ spaces of the Riesz transform is equivalent to $M$ having only one end. Therefore, the condition $b_1(\bar{M})=0$ is superfluous. 

Let us also mention that  boundedness results for the Riesz transform on forms on asymptotically conic manifolds are presented in the recent paper \cite{GS}.   

\bigskip

Finally, we prove a boundedness result for the Riesz transform on functions $d\Delta^{-1/2}$ that allows the presence of harmonic $1$-forms. It extends \cite[Theorem 5]{D2} to manifolds satisfying \eqref{d} and \eqref{UE}.

\begin{Thm}\label{Riesz}

Let $M$ be a complete non-compact connected manifold satisfying \eqref{d} and \eqref{UE}, and such that $\mathrm{Ric}_-$ satisfies \eqref{condik}. Assume also that \eqref{dd} holds for some $\kappa>3$. Then the Riesz transform $d\Delta^{-1/2}$ is bounded on $L^p$ for all $p\in (1,\kappa)$.

\end{Thm} 
In particular, in the case where $M$ is the connected sum of two Euclidean spaces $\R^n$, $n>3$, one recovers the main result of \cite{CCH}.

\begin{Rem}\label{nogent}

{\em

As a consequence of Theorem \ref{Riesz}, if $M$ satisfies \eqref{d} and \eqref{UE}, and is such that $\mathrm{Ric}_-$ satisfies \eqref{condik}, then the Riesz transform is bounded on $L^p$, for all $p\in (1,\delta)$, where

$$\delta=\sup\{\kappa\,;\,\eqref{dd}\mbox{ holds}\},$$
provided $\delta>3$. If furthermore $M$ satisfies the local Poincar\'e inequality \eqref{P_loc}, it follows from \cite{D3} that $\delta$ is equal to the parabolic dimension of $M$, that is the infimum of $p$'s such that $M$ is $p$-parabolic. The interval $(1,\delta)$ for the boundedness of the Riesz transform is actually sharp, without further assumption on $M$, as follows from \cite[Theorem B]{C1}. We point out that the parabolic dimension has recently proved to be a relevant exponent in problems related to the Riesz transform: see \cite{D1} (gluing results for the Riesz transform), \cite{C1} (Riesz transform on manifolds having quadratic curvature decay), and \cite{D3}  (Riesz transform of Schr\"{o}dinger operators).

}

\end{Rem}

\subsection{Strategy for the proof of the main result}

Let us now explain the ideas involved in the proof of Theorem \ref{main}. First, the implication i)$\Rightarrow$ iii) is easy enough and we give its proof right away:

\begin{Lem}\label{converse}
Let $M$ be a non-compact Riemannian manifold satisfying \eqref{d}, $E$  a vector bundle with  basis $M$, and $\mathcal{L}$  a generalised Schr\"{o}dinger operator on $E$. If \eqref{UEL} holds, then, for every $1 \leq p < +\infty$,
$$\mathrm{Ker}_{L^p}(\mathcal{L})=\{0\}.$$
\end{Lem}

\begin{proof}

Assume \eqref{UEL}. As a consequence, $e^{-t\mathcal{L}}$ is uniformly bounded on $L^p$, for all $p\in [1,\infty]$. Fix $p\in [1,\infty)$, and assume $\omega\in {\mathcal D}_{L^p}({\mathcal L})$ satisfies $\mathcal{L}\omega=0$. Then

$$\frac{d}{dt}\left(e^{-t\mathcal{L}}\omega\right)=0,$$
in $L^p$, which implies that for every $t\geq0$,

$$e^{-t\mathcal{L}}\omega=\omega.$$
Let $x\in M$. By \eqref{UEL},

$$|\omega(x)|=|e^{-t\mathcal{L}}\omega(x)|\leq \frac{C}{V^{1/p}(x,\sqrt{t})}||\omega||_p.$$
 
Letting $t\to +\infty$, one concludes that $\omega(x)=0$. Since $x$ is arbitrary, one gets $\omega\equiv0$. This shows that $\mathrm{Ker}_{L^p}(\mathcal{L})=\{0\}$. 
\end{proof}
The implications from iii) to ii) and from ii) to  i)  will rely in a crucial way on a decomposition of the operator $\mathcal{L}$:
we will consider  $\mathcal{L}$  as the rough Laplace operator $\bar{\Delta}$, plus $\mathcal{R}_+$, minus $\mathcal{R}_-$ outside a compact set (which is small in some sense thanks to condition $(K)$) perturbed by a compactly supported part of $\mathcal{R}_-$.

More precisely, let  $K_0$ be given by condition $(K)$. Let $\mathcal{W}_0$ and $\mathcal{W}_\infty$ be the sections of the vector bundle $\mathrm{End}(E)$ respectively given by $$x\to \mathcal{W}_0(x)=\mathbf{1}_{K_0}(x)\mathcal{R}_-(x)$$
and
$$x\to \mathcal{W}_\infty(x)=\mathbf{1}_{M\setminus {K_0}}\mathcal{R}_-(x).$$  We shall also denote by $\mathcal{W}_0$  and $\mathcal{W}_\infty$ the associated operators on sections of $E$.  Set
\begin{equation}\label{defh}
\mathcal{H}=\bar{\Delta}+\mathcal{R}_+-\mathcal{W}_\infty,
\end{equation}
so that 
$$\mathcal{L}=\mathcal{H}-\mathcal{W}_0.$$
That is, $\mathcal{L}$ can be seen as $\mathcal{H}$, perturbed by the compactly supported $\mathcal{W}_0$. We will sometimes write $\mathcal{W}_0^{1/2}$: by that, we mean that at every point $x$ in $M$, we take the square root of the non-negative, self-adjoint endomorphism $\mathcal{W}_0(x)$ of $E_x$, the fiber at $x$. Let us now detail the ideas of the proof of iii)$\Rightarrow$ ii) and ii)$\Rightarrow$ i) of Theorem \ref{main}. Here and in the rest of the paper,  $\|T\|_{p\to q}$, $1\le p,q\le +\infty$, will denote the norm of a linear operator $T$ from $L^p(E,\mu)$ to
 $L^q(E,\mu)$, $E$ being a vector bundle on $M$. 


\begin{itemize}

\item The implication ii)$\Rightarrow$ i) entails five steps:

\begin{enumerate}

\item By domination theory and condition $(K)$,  one  transfers Gaussian estimates from $(e^{-t\Delta})_{t>0}$  to $(e^{-t\mathcal{H}})_{t>0}$ (Proposition \ref{goh}). It follows by \cite{BCS} that
$$\sup_{t>0}||(I+t\mathcal{H})^{-1}V_{\sqrt{t}}^{1/p}||_{p\to\infty}<+\infty,$$
for all $p>\nu/2$. 

\item By a general functional analytic statement (Theorem \ref{blabla}) in the spirit of \cite{BCS}, one sees that the same resolvent estimates for $\mathcal{L}$ instead of $\mathcal{H}$ would yield Gaussian estimates for $(e^{-t\mathcal{L}})_{t>0}$, i.e. i).

\item Thanks to the perturbation formula

$$
(1+t\mathcal{L})^{-1}=(I-(1+t\mathcal{H})^{-1}t\mathcal{W}_0)^{-1}(1+t\mathcal{H})^{-1},
$$
one passes from the above resolvent  estimates on $\mathcal{H}$ to the desired resolvent  estimates on $\mathcal{L}$  
as soon as
$$\sup_{\lambda>0}||(I-(\mathcal{H}+\lambda)^{-1}\mathcal{W}_0)^{-1}||_{\infty\to\infty}<+\infty.$$
The latter fact is obtained in Proposition \ref{key}, whose proof consists in the following two steps (4) and (5).

\item By non-parabolicity of $M$, one deduces that $\mathcal{A}_0=\mathcal{W}_0^{1/2}\mathcal{H}^{-1}\mathcal{W}_0^{1/2}$  is compact on $L^2$, and under ii) $\mathcal{A}_0$ has spectral radius strictly less than one (see Lemma \ref{L4}).

\item Consider the operators
 $$\mathcal{B}_\lambda=(\mathcal{H}+\lambda)^{-1}\mathcal{W}_0,\quad \lambda\geq0.$$
Then, for all $\lambda\geq 0$, $\mathcal{B}_\lambda$ acting on $L^\infty$  is compact  (Lemma \ref{L1}) and its spectrum is related to the spectrum of $\mathcal{A}_0$. In particular, $\mathcal{B}_\lambda$ has spectral radius strictly less that one (Lemma \ref{L5}). By a Neumann series argument, this implies that $(I-(\mathcal{H}+\lambda)^{-1}\mathcal{W}_0)^{-1}$ is a  well-defined, uniformly bounded operator on $L^\infty$, which finishes the proof.

\end{enumerate}

\item For the implication iii)$\Rightarrow$ ii),  one has three steps: 

\begin{enumerate}

\item The first one is that, under \eqref{vol2} and non-parabolicity of $M$,  $\mathrm{Ker}_{L^p}(\mathcal{L})=\{0\}$ implies $\mathrm{Ker}_{H_0^1}(\mathcal{L})=\{0\}$, where $H_0^1$ is a natural Sobolev space associated to the quadratic form of $\mathcal{H}$. This is Lemma \ref{L7}.

\item The second one is the fact that if $M$ is non-parabolic, then $\mathrm{Ker}_{H_0^1}(\mathcal{L})=\{0\}$ is equivalent to the compact operator 
$\mathcal{A}_0$ having $L^2$ norm smaller than one
(see Lemma  \ref{L4}).

\item The estimate   $\|\mathcal{A}_0\|_{2\to 2}< 1$ discussed in point (2)  is equivalent to the strong subcriticality of $\mathcal{L}$, that is condition ii).

\end{enumerate}

\end{itemize}

\subsection{Plan of the paper}

In Section \ref{resemi}, we  present a general functional analytic statement in the spirit of \cite{BCS}, Theorem \ref{blabla}, which yields Gaussian estimates from resolvent estimates for a general semigroup. 

In Section 3, we study $\mathcal{L}$-harmonic sections. 

In Section 4, we prove the $L^p\to L^\infty$ estimates of the resolvent of $\mathcal{L}$ that are needed for the proof of Theorem \ref{main}.

In Section 5, we put the results of Sections 2, 3 and 4 together to conclude the proof of Theorem \ref{main}.

In Section 6, we prove some related results and consequences of Theorem \ref{main}, namely Theorem \ref{thm_para} and Corollary \ref{Riesz_hodge}.

Finally, in Section 7, we prove a consequence of Theorem \ref{main} for the Riesz transform, namely Theorem \ref{Riesz}.


\section{Resolvents and semigroups}\label{resemi}

This section will be devoted to a new version of the main result in \cite{BCS}, which is of independent interest. The new idea with respect to \cite{BCS} is that we iterate instead of extrapolating, which enables us to get rid of the assumption of uniform boundedness of the semigroup on $L^1$. For the heat kernel on functions the latter is given by the Markov property, but  in the case of forms it cannot be taken for
granted.

Let $\mathcal{L}$ be a generalised Schr\"odinger operator acting on sections of $E$. We recall some notation that has been introduced in \cite{BCS}: for $1\le p\le q\le \infty$ and $\gamma\in\R$, we consider
 
\begin{equation}\tag{$V\!RV_{p,q,\gamma}$}
\label{VRV}\sup_{t>0} \|V_{\sqrt{t}}^\gamma(I+t\mathcal{L})^{-1}V_{\sqrt{t}}^\delta  \|_{p\to q}  < +\infty
\end{equation}
and

 \begin{equation}\label{VEV}\tag{$V\!EV_{p,q,\gamma}$}
\sup_{t>0}\|V_{\sqrt{t}}^\gamma e^{-t\mathcal{L}}V_{\sqrt{t}}^\delta\|_{p\to q}<+\infty,
\end{equation}
where, in both estimates, $\delta$ is defined by $\gamma+\delta=\frac{1}{p}-\frac{1}{q}$. It follows from \cite[Corollary 2.1.7]{BCS} that if $e^{-t\mathcal{L}}$ has Gaussian estimates \eqref{UEL}, then both \eqref{VRV} and \eqref{VEV} hold for all admissible values of the parameters. Using an extrapolation argument, \cite[Theorem 1.2.1]{BCS} provides a converse to that: it asserts that if \eqref{VEV} (resp. \eqref{VRV}) are satisfied for $p=2$ and {\em some} $q>2,\gamma\in\R$ (resp. $q>2,\gamma\in R$ with $\frac{2q}{q-2}>\nu$), and if moreover $e^{-t\mathcal{L}}$ is {\em uniformly bounded} on $L^1$, then \eqref{UEL} holds. However, as we have already pointed out,  in our setting, due to the presence of $\mathcal{R}_-$ it cannot be assumed {\em a priori} that $e^{-t\mathcal{L}}$ is uniformly bounded on $L^1$. Nonetheless, we are able to overcome this difficulty by using an iterative argument (instead of the usual extrapolation argument), which first appeared in \cite{D2}. 

Note that in the following proof, we are going to use several results from \cite[Section 4]{BCS}.  These statements are formally written for semigroups acting on functions. But the extension to self-adjoint semigroups acting on vector bundles is straightforward. The following statement  could easily be formulated in a more abstract setting, namely vector bundles on doubling metric measure spaces  endowed with operators satisfying the Davies-Gaffney estimates. We leave this to the reader and stick to the Riemannian case.

\begin{Thm}\label{blabla} Let $M$ be a complete Riemannian manifold  satisfying \eqref{d},  $E$ a vector bundle on $M$, and  $\mathcal{L}$  a generalised  Schr\"odinger operator on $E$. Let $1\le \hat{p}<+\infty$.
If
\begin{equation}\tag{$RV_{{p},\infty}$}
\label{defres}\sup_{t>0} \|(I+t\mathcal{L})^{-1}V_{\sqrt{t}}^{1/{p}}  \|_{{p}\to \infty}  < +\infty,
\end{equation}
for all $\hat{p}\leq p<+\infty$, then  \eqref{UEL} holds.
\end{Thm}

\noindent{\em Proof: } 
Choose  ${\bar{p}}=2(n-1) \geq \hat{p} $, when $n\geq 2$ is an integer. By assumption 
$(RV_{{{\bar{p}}},\infty})$ holds.
By  duality 
\begin{equation*}\tag{$V\!R_{1,\bar{p}'}$}
\sup_{t>0} \|V_{\sqrt{t}}^{1/{\bar{p}}} (I+t\mathcal{L})^{-1}  \|_{1\to {\bar{p}}' }  < +\infty,
\end{equation*}
and by complex interpolation for the family of operators 

$$T_z=V_{\sqrt{t}}^{\frac{z}{\bar{p}}+(1-z)}e^{-t\mathcal{L}}V_{\sqrt{t}}^{-z+\frac{1-z}{p}},\qquad 0\leq \Re(z)\leq 1,$$ 
one obtains (see also \cite[Proposition 2.1.5]{BCS})

\begin{equation}\tag{$V\!RV_{p,q,\gamma}$}
\label{stepi}\sup_{t>0} \|V_{\sqrt{t}}^\gamma(I+t\mathcal{L})^{-1}V_{\sqrt{t}}^\delta  \|_{p\to q}  < +\infty
\end{equation}
 for any $ p, q$ such that $1\leq  p\leq {\bar{p}}$, ${\bar{p}'} \leq q\leq+\infty$, $\frac{1}{ p}-\frac{1}{ q}=\gamma+\delta=\frac{1}{{\bar{p}}}$, and $\gamma=\frac{1}{({\bar{p}}-1) q}$.

Next
consider the finite sequence $2=p_1<p_2< \ldots< p_n=\infty$ such that 
$\frac{1}{p_k}-\frac{1}{p_{k+1}}= \frac{1}{{\bar{p}}}$, in others words $\frac{1}{p_k}=\frac{1}{2}- \frac{k-1}{{\bar{p}}}$,  for $k\in \{1,2,\ldots, n\}$.
We can choose $p=p_k$, $q=p_{k+1}$ and $\gamma=\gamma_k:=\frac{1}{{\bar{p}}-1}\left(\frac{1}{2}-\frac{k}{ {\bar{p}}}\right)$ in \eqref{stepi}, which yields 
\begin{equation}
\sup_{t>0}\|V_{\sqrt{t}}^{\gamma_k}(I+t\mathcal{L})^{-1}V_{\sqrt{t}}^{\delta_k} \|_{p_k\to p_{k+1}}<+\infty,\label{Dv}
\end{equation}
with  $\gamma_k+\delta_k=\frac{1}{\bar{p}}$,
for all $k\in \{1,\ldots n-1\}$.

Let the function $\Phi$ on $\R$ be even  and satisfy 
\begin{equation}
\sup_{\lambda\in \R} |(1+\lambda^2)^{n} \Phi(\lambda)| <  +\infty\label{sup}
\end{equation}
as well as $\mbox{ supp}\, \widehat\Phi \subset [-1,1] $.
 
We shall prove inductively that 
\begin{equation}
\label{iter}\sup_{t>0}\|(I+t\mathcal{L})^{n-k}\Phi\left(\sqrt {t\mathcal{L}}\right)V_{\sqrt{t}}^{\gamma_k'} \|_{2\to p_{k}}<+\infty,\tag{$I_k$}
\end{equation}
for  $k\in \{1,\ldots n\}$, where $\gamma_k'=1/2-1/p_{k}=\frac{k-1}{{\bar{p}}}$.

Note first that  $(I_1)$, that is
\begin{equation*}
\sup_{t>0}\|(I+t\mathcal{L})^{n-1}\Phi\left(\sqrt {t\mathcal{L}}\right) \|_{2\to 2}<+\infty,
\end{equation*}
 follows from \eqref{sup}  and the spectral theorem. 

Assume now $(I_k)$ for $k\in \{1,2,\ldots, n-1\}$.
By \cite[Proposition 4.1.1 and Lemma 4.1.3]{BCS}, 
\begin{eqnarray*}
&&\|(I+t\mathcal{L})^{n-(k+1)}\Phi(\sqrt {t\mathcal{L}})V_{\sqrt{t}}^{\gamma_{k+1}'} \|_{2\to p_{k+1}}\\&\le& \|V_{\sqrt{t}}^{\gamma_k}(I+t\mathcal{L})^{n-(k+1)}\Phi(\sqrt {t\mathcal{L}})V_{\sqrt{t}}^{\gamma'_{k+1}-\gamma_{k}} \|_{2\to p_{k+1}}.
\end{eqnarray*}

Write
\begin{eqnarray*}
&&\|V_{\sqrt{t}}^{\gamma_k}(I+t\mathcal{L})^{n-(k+1)}\Phi(\sqrt {t\mathcal{L}})V_{\sqrt{t}}^{\gamma'_{k+1}-\gamma_{k}} \|_{2\to p_{k+1}}\\
&\leq& \|V_{\sqrt{t}}^{\gamma_k}(I+t\mathcal{L})^{-1}V_{\sqrt{t}}^{\delta_k} \|_{p_k\to p_{k+1}} \|V_{\sqrt{t}}^{-\delta_k}(I+t\mathcal{L})^{n-k}\Phi(\sqrt {t\mathcal{L}})V_{\sqrt{t}}^{\gamma'_{k+1}-\gamma_{k}} \|_{2\to p_{k}}.
\end{eqnarray*}

Again by \cite[Proposition 4.1.1 and Lemma 4.1.3]{BCS} ,
\begin{eqnarray*}&&\|V_{\sqrt{t}}^{-\delta_k}(I+t\mathcal{L})^{n-k}\Phi(\sqrt {t\mathcal{L}})V_{\sqrt{t}}^{\gamma'_{k+1}-\gamma_{k}} \|_{2\to p_{k}}\\&\le& \|(I+t\mathcal{L})^{n-k}\Phi(\sqrt {t\mathcal{L}})V_{\sqrt{t}}^{\gamma'_{k+1}-\gamma_{k}-\delta_k}\|_{2\to p_{k}}\\&=& \|(I+t\mathcal{L})^{n-k}\Phi(\sqrt {t\mathcal{L}})V_{\sqrt{t}}^{\gamma'_{k+1}-\frac{1}{\bar{p}}}\|_{2\to p_{k}}\\&=& \|(I+t\mathcal{L})^{n-k}\Phi(\sqrt {t\mathcal{L}})V_{\sqrt{t}}^{\gamma'_{k}}\|_{2\to p_{k}},
\end{eqnarray*}
hence, by \eqref{Dv} and $(I_k)$, $(I_{k+1})$ follows.
One has therefore $(I_n)$, that is

$$\sup_{t>0}\|\Phi(\sqrt {t\mathcal{L}})V_{\sqrt{t}}^{1/2} \|_{2\to \infty} <  +\infty.$$
Applying again \cite[Proposition 4.1.1 and Lemma 4.1.3]{BCS}, we obtain

\begin{equation}\label{2-inf}
\sup_{t>0}\|V_{\sqrt{t}}^{1/2}\Phi(\sqrt {t\mathcal{L}}) \|_{2\to \infty} <  +\infty.
\end{equation}
Next we consider the function $\Phi=F_a$ defined as the Fourier transform  of the function $t \to (1-{t^2})^a_+$ as in \cite[Proposition 4.1.6]{BCS}. As in the proof of \cite[Proposition 4.1.6]{BCS}, write, by spectral theory,
$$
\int_0^{+\infty} s^{a+\frac{1}{2}}e^{-s/4}F_a(\sqrt{st\mathcal{L}}) \, ds=e^{-t\mathcal{L}}.
$$
Thus
\begin{eqnarray*}
 &&\|V_{\sqrt{t}}^{1/2}  e^{-t\mathcal{L}} \|_{2\to \infty}\le \int_0^{+\infty} s^{a+\frac{1}{2}}e^{-s/4}\|V_{\sqrt{t}}^{1/2}   F_a(\sqrt{ts\mathcal{L}})    \|_{2\to \infty} \,ds\\
 &\le &  C\int_0^{+\infty} s^{a+\frac{1}{2}}e^{-s/4}\Big(1+\frac{1}{\sqrt{s}}\Big)^{\kappa_v/2}\|V_{\sqrt{st}}^{1/2}   F_a(\sqrt{ts\mathcal{L}})    \|_{2\to \infty} \,ds,
  \end{eqnarray*}
  hence, for $a$ large enough,
  $$
 \sup_{t>0} \|{V_{\sqrt{t}}^{1/2}}   e^{-t\mathcal{L}}\|_{2\to \infty}
 \le C'\sup_{t>0}\|V_{\sqrt{t}}^{1/2}   F_a(\sqrt{t\mathcal{L}})     \|_{2\to \infty},$$
 which is finite by \eqref{2-inf}. Hence
$$
\sup_{t>0}\|V_{\sqrt{t}}^{{1/2}}\exp(-t\mathcal{\mathcal{L}})\|_{2\to \infty}<+\infty,
$$
 that is, by \cite[Corollary 2.1.2]{BCS}, \eqref{DUEL}. The fact that \eqref{UEL} holds now follows from Theorem \ref{off-diago} in the Appendix.

\cqfd
Actually, an adaptation of the proof of Theorem \ref{blabla} allows us to prove the following more general result:

\begin{Thm}\label{thm_VEV}

Let $M$ be a complete Riemannian manifold satisfying \eqref{d}, and $E$ a vector bundle over $M$. Let $\mathcal{L}$ be a generalised Schr\"{o}dinger operator acting on the sections of $E$, and assume that

\begin{equation}\tag{$RV_{p,q}$}
\label{stepq}\sup_{t>0} \|(I+t\mathcal{L})^{-1}V_{\sqrt{t}}^{\frac{1}{p}-\frac{1}{q}} \|_{p\to q}  < +\infty
\end{equation}
for some  $2\le p\leq q<+\infty$. Then $(V\!EV_{p_1,q_1,\gamma})$ holds for all $\gamma\in\R$, $q'\leq p_1<q_1\le q$.

\end{Thm}
Notice that, contrary to Theorem \ref{blabla}, \eqref{stepq} is assumed to hold only for {\em one} particular $p$, and not for all $\tilde{p}\le p<\infty$. The result of Theorem \ref{thm_VEV} might be useful for semigroups that do not act on the full range of $L^p$ spaces.

\begin{proof}

The proof uses similar ideas to those used in the proof of Theorem \ref{blabla}. We only indicate what has to be changed. Interpolating ($RV_{p,q}$) with ($V\!RV_{2,2,0}$) (which holds by the spectral theorem), one obtains ($V\!RV_{r(\theta),s(\theta),\gamma(\theta)}$) for all $\theta\in [0,1]$, where 

$$\frac{1}{r(\theta)}=\frac{1-\theta}{2}+\frac{\theta}{p},\quad\frac{1}{r(\theta)}-\frac{1}{s(\theta)}=\theta\left(\frac{1}{p}-\frac{1}{q}\right).$$
By duality and interpolation, one obtains ($V\!RV_{r,s,\gamma}$) for all $s(\theta)'\leq r\leq r(\theta)$, $\frac{1}{r}-\frac{1}{s}=\theta\left(\frac{1}{p}-\frac{1}{q}\right)$ and some $\gamma\in \R$ (depending on $r$). The iterative argument of the proof of Theorem \ref{blabla} now gives ($V\!EV_{2,r,\gamma}$), for all $\gamma\in\R$, $q'\leq r\le q$. By duality and interpolation, one finally gets $(V\!EV_{p_1,q_1,\gamma})$ for all $\gamma\in\R$, $q'\leq p_1<q_1\le q$.

\end{proof}



Let us now focus on the semigroup $(e^{-t\mathcal{H}})_{t>0}$.


\begin{Pro}\label{goh}
Assume that $M$ satisfies  \eqref{d} and \eqref{due}. 
Assume in addition that   $\mathcal{R}_-$ satisfies condition 
\eqref{condik} and let $\mathcal{H}$ be defined by \eqref{defh}. Then 
the heat kernel of $\mathcal{H}$ satisfies the Gaussian estimates $(U\!E_{\mathcal{H}})$. In particular, $\mathcal{H}$ is non-negative on $L^2$.
\end{Pro}

 \noindent{\em Proof: }
By domination theory (see \cite{HSU}),
\begin{equation}\label{dom}
|e^{-t\mathcal{H}}\omega|\leq e^{-t(\Delta-W_\infty)}|\omega|,
\end{equation} 
where $W_\infty(x)=\|\mathcal{W}_\infty(x)\|_x=\mathbf{1}_{M\setminus {K_0}}\|\mathcal{R}_-(x)\|_x$. 
It is therefore enough to prove Gaussian upper estimates for the semigroup $(e^{-t(\Delta-W_\infty)})_{t>0}$. For  $\lambda\geq 0$, 
  $$(\Delta+\lambda)^{-1}(x,y)=\int_0^{+\infty}e^{-\lambda t}p_t(x,y)\,dt \le G(x,y),$$
  hence
  $$
  \sup_{x\in M}\int_{M}(\Delta+\lambda)^{-1}(x,y)W_\infty(y)\,d\mu(y)\le  \sup_{x\in M}\int_{M\setminus {K_0}}G(x,y)\|\mathcal{R}_-(y)\|_y\,d\mu(y).
  $$
Therefore  condition \eqref{condik}
  yields
  $$
  \sup_{x\in M}\int_{M}(\Delta+\lambda)^{-1}(x,y)W_\infty(y)\,d\mu(y)<1,
  $$
  thus, as we already observed,
 $$
\sup_{\lambda\geq 0} \|(\Delta +\lambda)^{-1} W_\infty\|_{\infty \to \infty} <1. 
 $$
 It follows by a standard Neumann series argument that
 \begin{equation}\label{neuneu}
\sup_{t>0} \|(I-(I+t\Delta)^{-1}tW_\infty)^{-1}\|_{\infty \to \infty} <+\infty.
\end{equation}
Now write
 the  perturbation formula
 \begin{equation}\label{peper}
\left(I+t(\Delta-W_\infty)\right)^{-1}=(I-(I+t\Delta)^{-1}tW_\infty)^{-1}(I+t\Delta)^{-1}
\end{equation}
and observe that by  \eqref{due}
\begin{equation}\label{res}
\sup_{t>0} \|(I+t\Delta)^{-1}V_{\sqrt{t}}^{1/p}  \|_{p \to \infty}  < +\infty
\end{equation}
for any $p> \nu/2$ (see \cite{BCS}).
It follows from \eqref{neuneu}, \eqref{peper} and \eqref{res}  that
$$
\sup_{t>0} \|\left(I+t(\Delta-W_\infty)\right)^{-1}V_{\sqrt{t}}^{1/p}  \|_{p \to \infty}  < +\infty.
$$
for any $p> \nu/2$.  Theorem \ref{blabla} then yields Gaussian estimates for $(e^{-t(\Delta-W_\infty)})_{t>0}$, hence for $(e^{-t\mathcal{H}})_{t>0}$.

In particular, the operator $e^{-t\mathcal{H}}$ is uniformly  bounded  with respect to $t>0$ on $L^1$, $L^\infty$ hence on $L^2$. Now, denoting by $E_{\mathcal{H}}(-\varepsilon^{-1},-\varepsilon)$ the projection-valued measure on the part of the spectrum of $\mathcal{H}$ lying in the interval $(-\varepsilon^{-1},-\varepsilon)$, $\varepsilon>0$, one has by the spectral theorem that $||E_{\mathcal{H}}(-\varepsilon^{-1},-\varepsilon)||_{2\to 2}\leq 1$, and assuming that $E_{\mathcal{H}}(-\varepsilon^{-1},-\varepsilon)\neq 0$,
$$||e^{-t\mathcal{H}}E_{\mathcal{H}}(-\varepsilon^{-1},-\varepsilon)||_{2\to 2}\geq e^{\varepsilon t}\to\infty\hbox{ as }t\to\infty.$$
This implies that 

$$||e^{-t\mathcal{H}}||_{2\to 2}\geq e^{\varepsilon t}\to\infty\hbox{ as }t\to\infty,$$
which is a contradiction. Thus, $E_{\mathcal{H}}(-\varepsilon^{-1},-\varepsilon)= 0$, that is, the spectrum of $\mathcal{H}$ does not intersect the interval $(-\varepsilon^{-1},-\varepsilon)$. Since this is true for all $\varepsilon>0$, one concludes that $\mathcal{H}$ is non-negative.

\cqfd

\begin{Rem}

{\em

In fact, using condition \eqref{condik} and a Neumann series argument (see \cite[Lemma 3.5]{D3}), one can show that for $\varepsilon$ small enough, $\Delta-(1+\varepsilon)W_\infty$ is non-negative, which implies that $\Delta-W_\infty$ is subcritical. Proposition \ref{goh} then also follows from 
\cite[Theorem 4.1]{D3}.


}

\end{Rem}
As a straightforward consequence of Proposition \ref{goh} and \cite[Corollary 2.1.7]{BCS}, one obtains:

\begin{Lem}\label{H_op}

Assume that $M$ satisfies  \eqref{d} and \eqref{due}. 
Assume in addition that   $\mathcal{R}_-$ satisfies condition 
\eqref{condik} and let $\mathcal{H}$ be defined by \eqref{defh}. Then we have the following weighted estimates of the heat kernel and of the resolvent of $\mathcal{H}$:
\begin{equation}\label{gaga}\tag{$EV_{p,q}$}
\sup_{t>0} \|e^{-t\mathcal{H}}V_{\sqrt{t}}^\delta  \|_{p\to q}  < +\infty
\end{equation}
 for any $1\leq p\leq q\leq \infty$ and $\delta=\frac{1}{p}-\frac{1}{q}$, and
 
\begin{equation}\label{gogo}\tag{$RV_{p,q}$}
\sup_{t>0} \|(I+t\mathcal{H})^{-1}V_{\sqrt{t}}^\delta  \|_{p\to q}  < +\infty
\end{equation}
for any $1\leq p\leq q\leq + \infty$ such that $\delta=\frac{1}{p}-\frac{1}{q}< \frac{2}{\nu}$.
 
\end{Lem}
This will be used in the proof of Theorem \ref{main}.


\section{Strong subcriticality and $\mathcal{L}$-harmonic sections}\label{harm}

Introduce $Q_\mathcal{H}$, the quadratic form associated to $\mathcal{H}$:

$$Q_\mathcal{H}(\omega)=\int_M |\nabla \omega|^2+<\mathcal{R}_+\omega,\omega>-\int_{M\setminus K_0}\left(\mathcal{R}_-\omega,\omega\right),$$
and denote by $H_0^1$ the completion of $C_0^\infty(E)$ for the norm $\|\omega\|_{Q_\mathcal{H}}^2=Q_\mathcal{H}(\omega)$. Note that the space $H_0^1$ depends on $\mathcal{H}$, that is on $\mathcal{L}$ and on $K_0$.  Then define $\mathrm{Ker}_{H_0^1}(\mathcal{L})$ to be the set of $\omega\in H_0^1$, such that $\mathcal{L}\omega=0$ in the distribution sense. We claim that if $\varepsilon>0$ is small enough so that

$$\sup_{x\in M}\int_{M\setminus K_0}G(x,y)||\mathcal{R}_-(y)||_y\,dy<1-\varepsilon,$$
then for all $\omega\in C_0^\infty(E)$,

\begin{equation}\label{eps-pos}
Q_\mathcal{H}(\omega)\geq \varepsilon \int_M|\nabla \omega|^2.
\end{equation}
In order to prove this, consider the operator $\mathcal{P}=\nabla^*\nabla-(1-\varepsilon)^{-1}\mathcal{W}_\infty$. According to Proposition \ref{goh}, the heat kernel associated with $\mathcal{P}$ has Gaussian estimates, hence is non-negative. In terms of quadratic forms, this means that for all $\omega\in C_0^\infty(E)$,

\begin{equation}\label{eps-pos2}
(1-\varepsilon)\int_M |\nabla \omega|^2-\int_{M\setminus K_0}\left(\mathcal{R}_-\omega,\omega\right)\geq 0.
\end{equation}
Thus, \eqref{eps-pos} holds. Note that if $K_0=\emptyset$, this proves Remark \ref{first}.

Now, the non-parabolicity of $M$ is equivalent to the existence of a positive function $\rho$ such that, for every $u\in C_0^\infty(M)$,

\begin{equation}\label{NP}
\int_M\rho u^2\leq \int_M|\nabla u|^2
\end{equation}
(see e.g. \cite{PT}). Using \eqref{eps-pos}, \eqref{NP} and the Kato inequality (see e.g. \cite[Proposition 2.2]{HSU}) $|\nabla \omega|\geq |\nabla|\omega||$ a.e., one obtains for all $\omega\in C_0^\infty(E)$,

\begin{equation}\label{vNP}
\int_M\rho |\omega|^2\leq Q_\mathcal{H}(\omega).
\end{equation}
In particular, this shows that $H_0^1$ injects into $L^2_{loc}$ so it is a function space. Note that as consequence of \eqref{vNP}, $\mathrm{Ker}_{L^2}(\mathcal{H})=\{0\}$. Indeed, by self-adjointness, every element $\omega$ of $\mathrm{Ker}_{L^2}(\mathcal{H})$ lies in the domain of the quadratic form $Q_\mathcal{H}$, and satisfy $Q_\mathcal{H}(\omega)=0$, which by \eqref{vNP} implies that $\omega\equiv0$. This allows us to define operators $\mathcal{H}^{-\alpha}$, $\alpha>0$, by means of the spectral theorem, namely as $f(\mathcal{H})$, where $f(x)=x^{-\alpha}$ for $x>0$, and $f(0)=0$; since $\mathrm{Ker}_{L^2}(\mathcal{H})=\{0\}$, the spectral measure does not charge $0$, and so the value of $f$ at $0$ does not matter. Equivalently, one can use the heat kernel to define $\mathcal{H}^{-\alpha}$:

$$\mathcal{H}^{-\alpha}=\frac{1}{\Gamma(\alpha)}\int_0^\infty t^{\alpha-1}e^{-t\mathcal{H}}\,dt.$$
In the case  $\alpha=\frac{1}{2}$, there is  yet another equivalent way to define $\mathcal{H}^{-1/2}$: the operator $\mathcal{H}^{-1/2}$, defined by the spectral theorem, is an isometric embedding from $C_0^\infty$ endowed with the $L^2$ norm, to $H_0^1$, and thus extends by density to a bounded operator from $L^2$ to $H_0^1$. It turns out that this operator is a bijective isometry from $L^2$ to $H_0^1$, and on $H_0^1\cap L^2$ it coincides with the operator $\mathcal{H}^{-1/2}$ defined by the spectral theorem. See \cite[Section 3]{D0} for more details; the proofs are written for the scalar Laplacian, but use only the self-adjointness and the inequality \eqref{NP}, so by \eqref{vNP} they easily adapt to the present context.

Note that under condition \eqref{vol1} (which we recall under assumptions \eqref{d} and \eqref{UE} holds if and only if $M$ is non-parabolic), the operator $\mathcal{H}^{-1}$ has a well-defined, finite kernel outside of the diagonal: indeed, using the fact that the heat kernel of $\mathcal{H}$ has Gaussian estimates, and the formula
$$\mathcal{H}^{-1}=\int_0^{+\infty} e^{-t\mathcal{H}}\,dt,$$
it follows that for every $x\neq y$,
$$||\mathcal{H}^{-1}(x,y)||_{y,x}\leq C\int_{d^2(x,y)}^{+\infty} \frac{dt}{V(x,\sqrt{t})}<+\infty.$$

\bigskip

\begin{Lem}\label{L7}

Assume $M$ is non-parabolic and satisfies condition \eqref{vol2} for some $p\in[2,+\infty]$. Then
$$\mathrm{Ker}_{L^2}(\mathcal{L})\subset \mathrm{Ker}_{H_0^1}(\mathcal{L})\subset \mathrm{Ker}_{L^p}(\mathcal{L}).$$
In particular, if  $M$ is non-parabolic and satisfies $(V^2)$, then
$$\mathrm{Ker}_{H_0^1}(\mathcal{L})=\mathrm{Ker}_{L^2}(\mathcal{L}).$$
\end{Lem}

\noindent{\em Proof: }
Let $\omega\in \mathrm{Ker}_{L^2}(\mathcal{L})$. By elliptic regularity, $\omega$ is $C^{1,\alpha}_{loc}$ for some $\alpha\in (0,1)$. Then, in the distribution sense,
$$\mathcal{H}\omega=\mathcal{W}_0\,\omega,$$
and since $\mathcal{W}_0$ has compact support, $\mathcal{H}\omega\in L^2$. Hence, $\omega$ is in the domain of the quadratic form $Q_\mathcal{H}$ associated to $\mathcal{H}$. Therefore (see \cite[Proposition 3.2]{D0}) $\omega \in H_0^1$. We conclude that $\mathrm{Ker}_{L^2}(\mathcal{L})\subset \mathrm{Ker}_{H_0^1}(\mathcal{L})$.\\

We now want to prove the inclusion $ \mathrm{Ker}_{H_0^1}(\mathcal{L})\subset \mathrm{Ker}_{L^p}(\mathcal{L})$ if \eqref{vol2} holds.  Let $\omega\in \mathrm{Ker}_{H_0^1}(\mathcal{L})$. Again,  by elliptic regularity, $\omega$ is $C^{1,\alpha}_{loc}$ for all $\alpha\in (0,1)$. Define

$$\tilde{\omega}=-\mathcal{H}^{-1}\mathcal{W}_0\,\omega.$$
We claim that $\omega=\tilde{\omega}$. That is to say, $\omega$ is the solution of the equation $\mathcal{H}u=-\mathcal{W}_0\,\omega$ of ``minimal growth'' at infinity. To see this, notice that since $\omega$ belongs to $\mathrm{Ker}_{H_0^1}(\mathcal{L})$, the following holds in the distribution sense:

$$\mathcal{H}(\omega-\tilde{\omega})=0.$$
Writing

$$\tilde{\omega}=\mathcal{H}^{-1/2}(\mathcal{H}^{-1/2}\mathcal{W}_0^{1/2})\mathcal{W}_0^{1/2}\omega,$$
and using the facts that $\mathcal{H}^{-1/2}:L^2\to H_0^1$, that $\mathcal{H}^{-1/2}\mathcal{W}_0^{1/2}$ is bounded on $L^2$ (from the fact that $\mathcal{W}_0$ has compact support and \cite[Prop. 3.4]{D0}, see also Lemma \ref{L4b}), and that $\mathcal{W}_0^{1/2}\omega\in L^2$, one gets that $\tilde{\omega}\in H_0^1$. Therefore, for every $\varphi\in C_0^\infty(E)$,

$$\begin{array}{rcl}
0&=&\int_M\left( \omega-\tilde{\omega},\mathcal{H}\varphi\right)\\
&=&\int_M\left( \mathcal{H}^{1/2}(\omega-\tilde{\omega}),\mathcal{H}^{1/2}\varphi\right)\\
&=&\langle \omega -\tilde{\omega},\varphi\rangle_{H_0^1}
\end{array}$$
Since by definition $C_0^\infty(E)$ is dense in $H_0^1$, one gets

$$\omega =\tilde{\omega}$$
in $H_0^1$. But by \eqref{vNP}, $H_0^1$ injects into $L^2_{loc}$, 
hence $\omega =\tilde{\omega}$ almost everywhere. Since $\omega$ and $\tilde{\omega}$ are both $C^{1,\alpha}_{loc}$, one gets that the equality $\omega =\tilde{\omega}$ holds pointwise on $M$. Thus,

\begin{equation}\label{omeg}
\omega=-\mathcal{H}^{-1}\mathcal{W}_0\,\omega.
\end{equation}
We claim that since $\mathcal{W}_0$ has compact support, by \eqref{vol2} the operator $\mathcal{H}^{-1}\mathcal{W}_0^{1/2}$ is bounded on $L^p$. Indeed,
$$\mathcal{H}^{-1}\mathcal{W}_0^{1/2}=\int_0^{+\infty}  e^{-t\mathcal{H}}\mathcal{W}_0^{1/2}\,dt,$$
hence
$$\|\mathcal{H}^{-1}\mathcal{W}_0^{1/2}\|_{p\to p} \le  
\int_0^1 \| e^{-t\mathcal{H}}\mathcal{W}_0^{1/2}\|_{p\to p}\,dt 
 +\int_1^{+\infty} \| e^{-t\mathcal{H}}\mathcal{W}_0^{1/2}\|_{p\to p}\,dt.$$
Since $\mathcal{W}_0^{1/2}\in L^\infty_{loc}$, $\|\mathcal{W}_0^{1/2}\|_{p\to p}<+\infty$ and
$$\int_0^1 \| e^{-t\mathcal{H}}\mathcal{W}_0^{1/2}\|_{p\to p}\,dt\le C
\int_0^1 \| e^{-t\mathcal{H}}\|_{p\to p}\,dt<+\infty$$
thanks to Proposition  \ref{goh}.
Now, since $\mathcal{W}_0$ is supported in $K_0$,
 $$\int_1^{+\infty} \| e^{-t\mathcal{H}}\mathcal{W}_0^{1/2}\|_{p\to p}\,dt
\le\int_1^{+\infty} \| e^{-t\mathcal{H}}\|_{L^1(K_0)\to L^p}\|\mathcal{W}_0^{1/2} \|_{p\to 1}\,dt.$$
By \eqref{gaga}, for a fixed $x_0\in K_0$,
$$
 \| e^{-t\mathcal{H}}\|_{L^1(K_0)\to L^p}\le  \frac{C}   {\left[V(x_0, \sqrt{t}) \right]^{1-\frac{1}{p}}}. $$
 Finally $\|\mathcal{H}^{-1}\mathcal{W}_0^{1/2}\|_{p\to p}<+\infty$ by
 \eqref{vol2}.
 
Now if $p=2$ this clearly finishes the proof. To conclude the argument for  $p>2$ we note that 
 $\omega\in L^\infty_{loc}$ (since $\omega$ is $C^{1,\alpha}_{loc}$ for all $\alpha\in (0,1)$) and since $\mathcal{W}_0$ has compact support, one gets that $\mathcal{W}_0^{1/2}\omega$ is in particular in $L^p$. Therefore we conclude that $\tilde{\omega}$ is in $L^p$. By \eqref{omeg}, we conclude that $\omega$ is in $L^p$. Consequently, we have shown that
$$\mathrm{Ker}_{H_0^1}(\mathcal{L})\subset L^p.$$

\cqfd
We now prove the result that was claimed in Remark \ref{optimal2}:

\begin{Pro}\label{exa}

Let $u\in C^\infty(\R^3)$ such that $u(x)=|x|^{-1}$ for all $|x|\geq 1$,
$u(x) > 0$ for all $x\in \R^3$, and let $V=\frac{\Delta u}{u}$, and $L=\Delta-V$. Then $L$ is non-negative, does not satisfy \eqref{vecUE}, and 

$$\mathrm{Ker}_{L^p}(L)=\left\{\begin{array}{lcr}
\{0\},\,\mbox{for }1\leq p\leq 3,\\
\neq \{0\}\,\mbox{for }p>3.
\end{array}\right.$$ 

\end{Pro}

\begin{proof}

Notice that $u$ is a positive solution of $Lv=0$. Therefore, by the Agmon-Allegretto-Piepenbrink theorem (see \cite{Ag}), $L$ is non-negative. Clearly, $u\in \mathrm{Ker}_{L^p}(\mathcal{L})$ for every $p>3$, which shows that for $p>3$, $\mathrm{Ker}_{L^p}(L)\neq \{0\}$. By Lemma \ref{converse}, $\mathrm{Ker}_{L^p}(L)=\{0\}$ for all $p\in [1,\infty)$ is necessary for the validity of \eqref{vecUE}; we can thus conclude that \eqref{vecUE} does not hold for $L$. It remains to show that for every $1\leq p\leq 3$, $\mathrm{Ker}_{L^p}(L)=\{0\}$. First, notice that from the ultracontractivity bound $||e^{-t\Delta}||_{p\to \infty}\leq Ct^{-\frac{3}{2p}},$ one has

$$||e^{-tL}||_{p\to \infty}\leq Ce^{Bt}t^{-\frac{3}{2p}},$$
where $B$ is an upper bound for $V$. In particular, it is finite for any fixed $t>0$. Since any element of $\mathrm{Ker}_{L^p}(L)$ is invariant under $e^{-tL}$, this shows that 

$$\mathrm{Ker}_{L^p}(L)\subset L^p\cap L^\infty.$$
It is thus enough to prove that $\mathrm{Ker}_{L^3}(L)=\{0\}$. Let $v\in\mathrm{Ker}_{L^3}(L)$. Define

$$g=\Delta^{-1}Vv.$$
Then $g$ satisfies $\Delta g=Vv$, therefore $\Delta(v-g)=0$. Since $V$ is smooth and compactly supported, $Vv$ is in $L^2$. On $\R^3$, $\Delta^{-1}$ is bounded from $L^2$ to $L^6$, and therefore $g-v\in L^6$. But by a result of Yau \cite{Y}, for any $q\in (0,\infty)$, there is no non-zero $L^q$ harmonic function on a complete manifold. As a consequence, $v=g$. Using the fact that the Green function on $\R^3$ is (up to a multiplicative constant) $|x-y|^{-1}$, one finds that if $\int_M V(y)v(y)\,dy\neq 0$ then as $|x|\to\infty$,

\begin{equation}\label{x-1}
g(x)\sim |x|^{-1}.
\end{equation}
Also, if  $\int_M V(y)v(y)\,dy= 0$, then as $|x|\to\infty$,

\begin{equation}\label{x-2}
g(x)\lesssim |x|^{-2},
\end{equation}
If \eqref{x-1} is satisfied, $v=g$ cannot be in $L^3$, contradiction. Therefore, \eqref{x-2} is satisfied, and this implies that 

$$\lim_{|x|\to\infty}\frac{v(x)}{u(x)}=0.$$
We are going to prove that $v\equiv0$. Let $\varepsilon>0$. Take $R>0$ big enough so that for all $|x|=R$,

\begin{equation}\label{boundary}
|v(x)|\leq \varepsilon u(x).
\end{equation}
Since $\mathcal{L}$ is non-negative, it satisfies the comparison principle on the ball of center $0$ and radius $R$, and given that $v$ and $u$ are both solutions of $Lw=0$, one deduces from \eqref{boundary} that for every $|x|\leq R$,

$$|v(x)|\leq \varepsilon ux).$$
Fix $x\in \R^3$, then for every $\varepsilon$ small enough so that $R>|x|$, one has

$$|v(x)|\leq \varepsilon u(x).$$
Letting $\varepsilon\to0$, hence $R\to\infty$, one finds that $v(x)=0$. Since this is true for all $x\in \R^3$, $v\equiv0$. As a consequence,

$$\mathrm{Ker}_{L^3}(L)=\{0\}.$$

\end{proof}
To conclude this section, we prove the result that was claimed in Remark \ref{two-planes}.

\begin{Pro}\label{Lpara}

Let $M=\R^2\sharp\R^2$, be a connected sum of two Euclidean planes. Then the Hodge Laplacian on $1$-forms $\vec{\Delta}$ is not strongly subcritical, and for every $p>2$,

$$\mathrm{Ker}_{L^p}(\vec{\Delta})\neq\{0\}.$$
If moreover the gluing is made in such a way that $M$ has genus zero, then

$$\mathrm{Ker}_{L^2}(\vec{\Delta})=\{0\}.$$

\end{Pro}

\begin{proof}

It follows from a result of Magniez \cite{M} that since $\nu=2$ for $M$, if $\vec{\Delta}$ were strongly subcritical, then $d^*\vec{\Delta}^{-1/2}$ would be bounded on $L^p$ for all $p\in (1,2]$. This would imply by duality that the Riesz transform $d\Delta^{-1/2}$ on functions is bounded on $L^p$, $p\in [2,\infty)$, which is known to be false (see \cite{CCH}). Thus, $\vec{\Delta}$ cannot be strongly subcritical.

Let us now prove the two statements concerning $\mathrm{Ker}_{L^p}(\vec{\Delta})$. We first observe that there is a compact set $K$ such that $M\setminus K$ is isometric to two copies of $\R^2\setminus \mathbb{D}$, where $\mathbb{D}$ is the disk of center zero and radius, say, $1$. Applying the M\"{o}bius transformation $f(z)=1/z$ on each copy of $\R^2\setminus \mathbb{D}$, we see that $M\setminus K$ is conformally equivalent to the disjoint union of two Euclidean unit disks, punctured at their center. We can thus compactify $M$ by adding the two punctures $p$ and $q$, and obtain a compact Riemann surface $(\Sigma,g)$ such that $M$ is conformal to $(\Sigma\setminus \{p,q\},g)$. Moreover, the metric $g$ coincide with the standard spherical metric on a neighborhood of $p$ and of $q$. By general theory of Riemann surfaces (see e.g. \cite{FK}, p.51), there exists a meromorphic (complex) differential $\eta$ on $\Sigma^2$, that is holomorphic on $\Sigma^2\setminus \{p,q\}$ and has singularity $\frac{dz}{z}$ at $p$, and $-\frac{dz}{z}$ at $q$. It gives rise to a meromorphic differential $\tilde{\eta}$ on $M$, with the following asymptotics:

\begin{equation}\label{asymp}
\tilde{\eta}(\xi)\sim \mp\frac{d\xi}{\xi},\qquad\xi\to\pm\infty_M,
\end{equation}
where ``going to $+\infty_M$'' means going to infinity in the copy of $\R^2$ corresponding to $p$, and ``going to $-\infty_M$'' means going to infinity in the copy of $\R^2$ corresponding to $q$. Considering $\omega=\tilde{\eta}+\bar{\tilde{\eta}}$, one obtains a (real) harmonic $1$-form on $M$, and from \eqref{asymp}, it is obvious that $\omega\in L^p$, for all $p>2$. Thus, we have proved that

$$\mathrm{Ker}_{L^p}(\vec{\Delta})\neq\{0\},\,p>2.$$

We finally prove the statement concerning $\mathrm{Ker}_{L^2}(\vec{\Delta})$. The fact that $M$ has genus zero means that $\Sigma$ is topologically a two-dimensional sphere $\mathbb{S}^2$. Thus, $M$ is conformal to a sphere with two punctures $(\mathbb{S}^2\setminus \{p,q\},g)$.  By conformal invariance of the Dirichlet integral, every harmonic $1$-form $\omega$ on $M$ with a finite $L^2$ norm gives rise to a harmonic $1$-form $\tilde{\omega}$ on $\mathbb{S}^2\setminus \{p,q\}$. Since $\tilde{\omega}$ is in $L^2$, by Weyl's lemma its singularities at $p$ and $q$ are removable, and thus $\tilde{\omega}$ extends to a $L^2$ harmonic $1$-form on $(\mathbb{S}^2,g)$. Since $H^1(\mathbb{S}^2,\R)$, the first De Rham cohomology group of $\mathbb{S}^2$ is trivial, the Hodge theorem implies that $\tilde{\omega}\equiv0$, and thus $\omega\equiv0$. This proves that $\mathrm{Ker}_{L^2}(M)=\{0\}$.

\end{proof}

\section{Spectral radius estimates}\label{spectral}

 In order to obtain Gaussian estimates for $(e^{-t\mathcal{L}})_{t>0}$, a crucial step will be to pass from \eqref{gogo} to a similar estimate for  $\mathcal{L}$,
by writing,  as in \cite{D2}, the standard perturbation formula:
$$
(1+t\mathcal{L})^{-1}=(I-(1+t\mathcal{H})^{-1}t\mathcal{W}_0)^{-1}(1+t\mathcal{H})^{-1},
$$
that is, passing to the resolvent by setting $\lambda=1/t$,
\begin{equation}\label{reso}
(\mathcal{L}+\lambda)^{-1}=(I-(\mathcal{H}+\lambda)^{-1}\mathcal{W}_0)^{-1}(\mathcal{H}+\lambda)^{-1}.
\end{equation}
The term $(\mathcal{H}+\lambda)^{-1}$ is taken care of by Proposition \ref{goh}. We now have to study the factor $(I-(\mathcal{H}+\lambda)^{-1}\mathcal{W}_0)^{-1}$. The key to prove Theorem \ref{main} is the following:

\begin{Pro}\label{key}

Assume that $M$ satisfies \eqref{d}, \eqref{UE} and \eqref{vol1}, and that $\mathrm{Ker}_{H_0^1}(\mathcal{L})=\{0\}$. Then the operator $(I-(\mathcal{H}+\lambda)^{-1}\mathcal{W}_0)^{-1}$ is a well-defined bounded operator on $L^\infty$ and 
$$\sup_{\lambda>0}\|(I-(\mathcal{H}+\lambda)^{-1}\mathcal{W}_0)^{-1}\|_{\infty\to \infty}<+\infty.$$
\end{Pro}
To be able to prove Proposition \ref{key} first we have to discuss a few preliminary observations. For $\lambda\geq 0$, we introduce the two operators

$$\mathcal{A}_\lambda=\mathcal{W}_0^{1/2}(\mathcal{H}+\lambda)^{-1}\mathcal{W}_0^{1/2}$$
and

$$\mathcal{B}_\lambda=(\mathcal{H}+\lambda)^{-1}\mathcal{W}_0.$$
There are two main ideas for the proof of Proposition \ref{key}: the first one is that the $L^\infty$ spectrum of $\mathcal{B}_\lambda$ is related to the $L^2$ spectrum of $\mathcal{A}_\lambda$. The second one is that the spectral radius of $\mathcal{A}_\lambda$ being strictly less than $1$ is equivalent to the strong subcriticality of $\mathcal{L}$, which will be shown to be equivalent to the more geometric condition $\mathrm{Ker}_{H_0^1}(\mathcal{L})=\{0\}$. The first idea goes back to the seminal work of B. Simon on Schr\"{o}dinger operators in $\R^n$ (see \cite[Lemma 3.9]{S}), whereas the second one appears in \cite{D2}, for manifolds satisfying a global Sobolev inequality.\\

We now present a series of lemmas, which will be used in the proof of Proposition \ref{key}.

\begin{Lem}\label{L6} Assume that $M$  satisfies \eqref{d}, \eqref{UE} and \eqref{vol1}, and that $\mathcal{R}_-$ satisfies condition 
\eqref{condik}. Let $\mathcal{H}$ be defined by \eqref{defh}.  Then the kernel $(\mathcal{H}+\lambda)^{-1}(x,y)$ converges uniformly as $\lambda\to 0_+$ to $\mathcal{H}^{-1}(x,y)$ on compact sets of $M^2\setminus \mathrm{Diag}(M)$,
where $\mathrm{Diag}(M):=\{(x,x); x\in M\}$.
\end{Lem}

\noindent{\em Proof :} Notice first that 

$$\int_1^\infty \frac{dt}{V(x_0,\sqrt{t})}\geq \int_1^\infty \frac{dt}{t^{\nu/2}},$$
so \eqref{vol1} implies $\nu>2$. Let $\tilde K$ be a compact set of $M^2\setminus \mathrm{Diag}(M)$, and denote
$$d(\tilde K,\mathrm{Diag}(M))=\mu>0.$$
It is straightforward to check that $d(x,y)\geq \mu$ if $(x,y)\in \tilde K$.
Write
$$\mathcal{H}^{-1}(x,y)-(\mathcal{H}+\lambda)^{-1}(x,y)=\int_0^{+\infty} \left(1-e^{-\lambda t}\right)e^{-t\mathcal{H}}(x,y)\,dt,$$
By Proposition \ref{goh},
\begin{equation}\label{PSL}
\|e^{-t\mathcal{H}}(x,y)\|_{y,x}\leq \frac{C}{V(x,\sqrt{t})}e^{-\mu^2/Ct},
\end{equation}
for $(x,y)\in \tilde K$ and $t>0$.
But by \eqref{dnu},
$$V(x,\sqrt{t})\geq ct^{\nu/2}V(x,1)\mbox{ if }t\leq 1.$$
Let $\bar{K}$ be the projection of $\tilde K$ on the first coordinate.
Since $\bar{K}$ is compact, there is a constant $c>0$ such that if $(x,y)\in \tilde {K}$ then
$$V(x,1)\geq c,$$
therefore
$$V(x,\sqrt{t})\geq c't^{\nu/2}\mbox{ if }t\leq 1.$$
Now, for all $(x,y)\in \tilde K$,
\begin{eqnarray*}
\left\|\int_0^1 \left(1-e^{-\lambda t}\right)e^{-t\mathcal{H}}(x,y)\,dt\right\|_{y,x}&\le& \int_0^1 \left(1-e^{-\lambda t}\right)\|e^{-t\mathcal{H}}(x,y)\|_{y,x} \,dt\\
&\le& \int_0^1  \left(1-e^{-\lambda t}\right)t^{-\nu/2}e^{-\mu^2/Ct}dt<{+\infty},
\end{eqnarray*}
which   converges to $0$ as $\lambda\to0^+$ by dominated convergence since $$\int_0^1 t^{-\nu/2}e^{-\mu^2/Ct}dt<{+\infty}.$$
Now for the integral from $1$ to ${+\infty}$.  By  \eqref{PSL},
$$\|e^{-t\mathcal{H}}(x,y)\|_{y,x}\leq \frac{C}{V(x,\sqrt{t})}$$
 for $(x,y)\in \tilde K$ and $t \geq 1$. Fix $x_0\in M$. By \eqref{d} there exists   $C_{\bar{K},x_0}$  such that 
$$\sup_{x\in \bar{K}}\frac{1}{V(x,\sqrt{t})} \le \frac{C_{\bar{K},x_0}}{V(x_0,\sqrt{t})},$$
hence for all $(x,y)\in \tilde K$
$$\int_1^{+\infty} \left(1-e^{-\lambda t}\right)\|e^{-t\mathcal{H}}(x,y)\|_{y,x}\,dt \le C'_{\bar{K},x_0} 
\int_1^{+\infty}  \left(1-e^{-\lambda t}\right)\frac{1}{V(x_0,\sqrt{t})}\,dt,
$$
which converges to $0$ as $\lambda\to0^+$ by dominated convergence thanks to condition \eqref{vol1}.
 It shows that
 $$\|\mathcal{H}^{-1}(x,y)-(\mathcal{H}+\lambda)^{-1}(x,y)\|_{y,x}\to 0$$
 as $\lambda\to0^+$, uniformly for $(x,y)\in\tilde{K}$ and concludes the proof of Lemma \ref{L6}.

\cqfd

The following lemma is crucial for our study. It will be used first in the proof of Lemma \ref{L5}, and again directly in the proof of Proposition \ref{key}.

\begin{Lem}\label{L1}
Under the assumptions of Lemma $\ref{L6}$, for any $\lambda\in [0,{+\infty})$, $\mathcal{B}_\lambda$ is compact on $L^\infty$, 
$\sup_{\lambda\geq 0}\|\mathcal{B}_\lambda\|_{\infty\to \infty}<+\infty,$
and the map $\lambda\mapsto \mathcal{B}_\lambda\in \mathcal{L}(L^\infty,L^\infty)$ is continuous on $[0,\infty)$. 
\end{Lem}

For the proof of Lemma \ref{L1}, we shall use Lemma \ref{L6} as well as the following compactness criterion, which follows from the Arzel\`a-Ascoli theorem (see \cite[Corollary 5.1]{E}):


\begin{Lem}\label{C1}

Let $T$ be an integral operator with continuous kernel $k(x,y)$. Assume that there exists a compact subset $K$ of $M$ such that $k(x,y)$ is supported in $M\times K$, and

$$\lim_{R\to \infty}\sup_{x\in M\setminus B(x_0,R)}\int_K|k(x,y)|\,d\mu(y)=0,$$
for some $x_0\in M$. Then $T$ is compact on $L^\infty$.

\end{Lem}

\begin{proof}
Consider a family of functions $\eta_{R} \in C_c(M)$, for all  $R>1$  such that, $0 \le \eta_R(x) \le 1$
for all $x\in M$,  $\eta_R(x)=1 $ for all $x \in B(x_0,R)$ and $\eta_R(x)=0 $ for all $x \notin B(x_0,2R)$. Let $T_R$ be 
 an integral operator with continuous kernel $\eta_R(x)k(x,y)$. By hypothesis the norm $\|T-T_R\|_{\infty\to \infty}$ converges to zero when $R$ goes to infinity. Hence it is enough to show that the operators $T_R$ are  compact for all 
 $R>1$. But every  $T_R$ has 
 continuous, compactly supported kernel. If follows that the kernel is uniformly continuous and  the standard argument based on 
 the Arzel\`a-Ascoli theorem shows that 
 an operator with uniformly continuous 
 kernel is compact on $L^\infty$ (and in fact on any other $L^p$ spaces). See e.g.  \cite[Example 4.1]{Kato}.

\end{proof}

%
%
%
%

\noindent{\em Proof of Lemma $\ref{L1}$: }
Let $K_0$ be the compact set from condition \eqref{condik}. Decompose
$$\mathcal{B}_\lambda=\mathcal{T}_\lambda+\mathcal{S}_\lambda=\chi (\mathcal{H}+\lambda)^{-1}\mathcal{W}_0+(1-\chi)(\mathcal{H}+\lambda)^{-1}\mathcal{W}_0,$$
 where $\chi\in C_0^\infty(M)$ is a cut-off function such that $\chi\equiv1$ on $$K_1:=\{x\,,\,d(x,K_0)\leq1\}.$$ The reason why we introduce $\mathcal{T}_\lambda$ is the following: the kernel of $\mathcal{B}_\lambda$ is $C^{1,\alpha}_{loc}$, $\alpha\in (0,1)$, outside  the diagonal, and singular on the diagonal. The operator  $\mathcal{T}_\lambda$ captures the ``near-diagonal" part of $\mathcal{B}_\lambda$. \\
 
\noindent \underline{Step 1:} we start by showing that \begin{equation}\label{T_l}
\sup_{\lambda\geq0}\|\mathcal{T}_\lambda\|_{\infty\to \infty}<+\infty,
\end{equation}
$\mathcal{T}_\lambda$ is compact on $L^\infty$ for every $\lambda\geq0$, 
and the map $\lambda\mapsto \mathcal{T}_\lambda\in \mathcal{L}(L^\infty,L^\infty)$ is continuous at $\lambda=0$. Write
$$\mathcal{T}_\lambda=\int_0^{+\infty} e^{-t\lambda}\chi e^{-t\mathcal{H}} \mathcal{W}_0\,dt,$$
so that
$$\|\mathcal{T}_\lambda\|_{\infty\to \infty}\leq \int_0^{+\infty} \| \chi e^{-t\mathcal{H}} \mathcal{W}_0\|_{\infty\to \infty}\,dt.$$
We split the integral into $\int_0^1+\int_1^{+\infty}$, and estimate both terms. Since by Proposition \ref{goh}, $\|e^{-t\mathcal{H}}\|_{\infty\to \infty}\leq C$, so that
$$\int_0^1 \|\chi e^{-t\mathcal{H}}\mathcal{W}_0\|_{\infty\to \infty}\,dt<+\infty.$$
Also,
\begin{eqnarray*}
\int_1^{+\infty} &\|\chi e^{-t\mathcal{H}}\mathcal{W}_0\|_{\infty\to \infty}\,dt  \\
&\leq &\int_1^{+\infty} \|\chi\|_{\infty}\|e^{-t\mathcal{H}}\|_{L^1(K_0)\to L^{\infty}(K_1)}\,dt \times \int_{K_0} \|\mathcal{W}_0(x)\|_x\,d\mu(x).
\end{eqnarray*}
%
But 
$\int_{K_0} \|\mathcal{W}_0(x)\|_x\,d\mu(x)$ is finite since $\mathcal{W}_0\in L^\infty$.
Write
$$\|e^{-t\mathcal{H}}\|_{L^1(K_0)\to L^\infty(K_1)}=\sup_{x\in K_1,y\in K_0} \|e^{-t\mathcal{H}}(x,y)\|_{y,x}.$$
As in Lemma \ref{L6} 
$$\int_{1}^{+\infty} \sup_{x\in K_1,y\in K_0} \|e^{-t\mathcal{H}}(x,y)\|_{y,x}\,dt \leq C_{K_0,K_1} \int_{1}^{+\infty} \frac{1}{V(x_0,\sqrt{t})}\, dt< {+\infty}, $$
where $x_0\in K_1$ is a point that is fixed.
%
%
%
%
Therefore 

$$\int_1^{+\infty} \|\chi e^{-t\mathcal{H}}\mathcal{W}_0\|_{\infty \to \infty}\,dt<{+\infty}.$$
Consequently,

$$\sup_{\lambda\geq0}\|\mathcal{T}_\lambda\|_{\infty\to \infty}<{+\infty}.$$
%

Let us prove now that the map $\lambda\mapsto \mathcal{T}_\lambda\in \mathcal{L}(L^\infty,L^\infty)$ is continuous for all  $\lambda \in [0, {+\infty})$. We have

$$\|\mathcal{T}_{\lambda'}-\mathcal{T}_{\lambda''}\|_{\infty\to \infty}\leq \int_0^{+\infty} 
(e^{-\lambda' t}-e^{-\lambda'' t}) \|\chi e^{-t\mathcal{H}}\mathcal{W}_0\|_{\infty\to \infty}\,dt.$$
We have just seen that $\int_0^{+\infty} \|\chi e^{-t\mathcal{H}}\mathcal{W}_0\|_{\infty\to \infty}\,dt<+\infty$. It follows by the dominated convergence theorem that 
$$\lim_{\lambda'\to \lambda''}\|\mathcal{T}_{\lambda'}-\mathcal{T}_{\lambda''}\|_{\infty\to \infty}=0.$$
Let us now show that $\mathcal{T}_\lambda$ is compact on $L^\infty$. 

%
%
%

 Using local H\"{o}lder regularity estimates (see e.g. \cite[Theorem 8.24]{GT}) for solutions of $(\mathcal{H}+\lambda)u=\mathcal{W}_0 f$ and the fact that $(\mathcal{H}+\lambda)^{-1}\mathcal{W}_0$ sends $L^\infty$ to $L^\infty_{loc}$ (by the arguments in Step 1), one obtains in a similar way as above:
$$\|\mathcal{T}_\lambda f\|_{C^{\alpha}}\leq C\|f\|_\infty.$$
By the Arzel\`a-Ascoli theorem the embedding $C^{\alpha}_{loc}\hookrightarrow L^\infty$ is compact, and thus $\mathcal{T}_\lambda$ is compact on $L^\infty$. This concludes the first part of the proof of Lemma \ref{L1}.\\

\noindent \underline{Step 2:} let us now turn to $\mathcal{S}_\lambda$, and show that $\mathcal{S}_\lambda$ is a compact operator on $L^\infty$, that

$$\sup_{\lambda\geq0}\|\mathcal{S}_\lambda\|_{\infty\to \infty}<{+\infty},$$
and that the map $\lambda\mapsto \mathcal{S}_\lambda$ is continuous at $\lambda=0$. 
By Proposition \ref{goh}, the heat kernel of $\mathcal{H}$ has Gaussian estimates. Using the expression $(\mathcal{H}+\lambda)^{-1}=\int_0^{+\infty} e^{-\lambda t}e^{-t\mathcal{H}}\,dt$, it follows that the kernel of $(\mathcal{H}+\lambda)^{-1}$ can be estimated by:
$$||(\mathcal{H}+\lambda)^{-1}(x,y)||_{y,x}\leq C\int_{d^2(x,y)}^{+\infty}\frac{dt}{V(y,\sqrt{t})},$$
where $C$ is independent of $\lambda$. Now, 
$S_\lambda(x,y)=0$ unless $y\in K_0$ and $x\not\in K_1$, in which case $d(x,y)\geq 1$.
Thus
$$\|S_\lambda(x,y)\|_{y,x}\leq C\int_{1}^{+\infty}\frac{dt}{V(y,\sqrt{t})}\leq C'\int_{1}^{+\infty}\frac{dt}{V(y_0,\sqrt{t})},$$
where $y_0\in K_0$ is fixed. By \ref{vol1}, the above integral is finite.
%
Finally
$$||S_\lambda||_{\infty\to \infty}\leq C$$
for some constant $C$ independent of $\lambda$.

Also, since
$$\lim_{x\to \infty} ||S_\lambda(x,y)||_{y,x}=0,$$
uniformly for $y\in K_0$, and since by elliptic regularity $S_\lambda(x,y)$ is $C^{1,\alpha}$ (since it has support outside of the diagonal), Lemma \ref{C1} implies that $\mathcal{S}_\lambda$ is compact on $L^\infty$. Let us show that $\lambda\mapsto \mathcal{S}_\lambda$ is continuous at $\lambda=0$. First, notice that for every $\varepsilon>0$, there exists $K_0\Subset K_{\varepsilon} \Subset M$ such that for every $\lambda\geq0$,

$$\|\chi_{M\setminus K_\varepsilon}\mathcal{S}_\lambda\|_{\infty\to \infty} \leq \varepsilon.$$
Together with the fact that the kernel $(\mathcal{H}+\lambda)^{-1}(x,y)$ converges to $\mathcal{H}^{-1}(x,y)$ uniformly on compact sets of $M^2\setminus\mathrm{Diag}(M)$ (see Lemma \ref{L6}), this implies that $\lambda\mapsto \mathcal{S}_\lambda\in\mathcal{L}(L^\infty,L^\infty)$ is continuous at $\lambda=0$.\\

\noindent \underline{Step 3:} finally, let us prove the continuity of $\lambda\mapsto \mathcal{B}_\lambda\in\mathcal{L}(L^\infty,L^\infty)$ at a point $\lambda_0>0$. Write

$$\|(\mathcal{H}+\lambda_0)^{-1}\mathcal{W}_0-(\mathcal{H}+\lambda)^{-1}\mathcal{W}_0\|_{\infty\to \infty}=|\lambda-\lambda_0|\,\|(\mathcal{H}+\lambda)^{-1}(\mathcal{H}+\lambda_0)^{-1}\mathcal{W}_0\|_{\infty\to \infty}.$$
Take $|\lambda-\lambda_0|\leq \frac{|\lambda_0|}{2}$, then 

$$\|(\mathcal{H}+\lambda)^{-1}\|_{\infty\to \infty}\leq \frac{C}{\lambda}\leq \frac{2C}{\lambda_0}.$$
Since $\|(\mathcal{H}+\lambda_0)^{-1}\mathcal{W}_0\|_{\infty\to \infty}\leq C$, we get that for $\lambda$ close enough to $\lambda_0$,

$$\|(\mathcal{H}+\lambda)^{-1}(\mathcal{H}+\lambda_0)^{-1}\mathcal{W}_0\|_{\infty\to \infty}\leq C.$$
Thus,

$$\lim_{\lambda\to\lambda_0}\|(\mathcal{H}+\lambda_0)^{-1}\mathcal{W}_0-(\mathcal{H}+\lambda)^{-1}\mathcal{W}_0\|_{\infty\to \infty}=0,$$
which shows the claim.\\

Let us now present a lemma which is an adaptation of \cite[Corollary 3.1, Proposition 3.5]{D0} and of the proof of \cite[Theorem 3.1]{D0} to the case of generalised Schr\"{o}dinger operators acting on sections of a vector bundle:

\begin{Lem}\label{L4b}

Assume that $M$ is non-parabolic. Then, for all $\lambda\geq0$, the operators $(\mathcal{H}+\lambda)^{-1/2}\mathcal{W}_0$ and $\mathcal{W}_0(\mathcal{H}+\lambda)^{-1/2}$ are compact on $L^2$, and adjoint to each other. 

\end{Lem} 

\begin{proof}

We only explain how to prove the compactness in $L^2$ of the operator $\mathcal{U}_0= \mathcal{W}_0\mathcal{H}^{-1/2}$. Once this is done, the compactness of $\mathcal{W}_0(\mathcal{H}+\lambda)^{-1/2}$ follows from the formula

$$\mathcal{W}_0(\mathcal{H}+\lambda)^{-1/2}=\mathcal{U}_0\mathcal{H}^{1/2}(\mathcal{H}+\lambda)^{-1/2},$$
and from the fact that $(\mathcal{H}+\lambda)^{-1/2}$ is a contraction on $L^2$, by the spectral theorem. The other claims of the Lemma are then proved following \cite{D0}, using only functional analysis arguments that generalize easily to the case of generalised Schr\"{o}dinger operators, acting on sections of a vector bundle. Let us just mention that the fact that the operators $\mathcal{W}_0\mathcal{H}^{-1/2}$ and $\mathcal{H}^{-1/2}\mathcal{W}_0$ are adjoint one another is {\em not} completely obvious, because it is a non-trivial fact that $\mathcal{H}^{-1/2}\mathcal{W}_0$ is bounded on $L^2$. See \cite[Proposition 3.4]{D0} and the proof of \cite[Theorem 3.1]{D0} for more details. 

Now, let us come back to $\mathcal{U}_0$. First, by definition of $H_0^1$, 

$$\mathcal{H}^{-1/2} : L^2\rightarrow H_0^1$$
is an isometry. Next, recall that by \eqref{eps-pos} and \eqref{vNP}, 

$$H_0^1\hookrightarrow \tilde{W}^{1,2}_{loc},$$
where

$$\tilde{W}^{1,2}_{loc}:=\{\omega\in L^2_{loc}\,;\,\nabla \omega\in L^2_{loc}\}.$$
We claim that 

\begin{equation}\label{Sob}
\tilde{W}^{1,2}_{loc}=W^{1,2}_{loc},
\end{equation}
where $W^{1,2}_{loc}$ is defined as the space of $\omega$ with $\omega\circ \varphi\in W^{1,2}(\R^N)$, for every local trivialization $\varphi:B\times \R^m\subset \R^{n+m}\rightarrow E$ of the vector bundle $E\to M$. Here, $n$ is the dimension of $M$, $m$ is the dimension of a fiber of $E$, and $B$ is the unit ball in $\R^n$. Obviously,  $W^{1,2}_{loc}\subset \tilde{W}^{1,2}_{loc}$. For the converse, let us take a local trivialization $\varphi:B\times \R^m\subset \R^{n+m}\rightarrow E$ of $E\to M$, around a point $p\in M$; in the natural coordinates induced by $\varphi$, every section $\omega$ of $E$ writes $\omega=\sum_{j=1}^m\omega^je_j$, for some basis $e_1,\ldots,e_m$ of $\R^m$, and $\nabla \omega$ writes

$$\nabla \omega=\sum_{k=1}^m(d\omega^k+\sum_{j=1}^m A_j^k\omega^j)e_k,$$
where the matrix $(A_j^k(x))_{1\leq j,k\leq m}$ depends on the connection $\nabla$, and is smooth in $x\in M$. Assuming that $\omega$ is in $\tilde{W}^{1,2}_{loc}$, one has $A_j^k\omega^je_k\in L^2(B)$, and thus $d\omega^j\in L^2(B)$ for all $j$. Thus, $\omega$ is in $W^{1,2}_{loc}$, and we conclude that \eqref{Sob} holds. Therefore, one obtains

$$\mathcal{H}^{-1/2}:L^2\rightarrow H_0^1\hookrightarrow W^{1,2}_{loc}.$$
By the Rellich-Kondrachov theorem, the embedding
$$W^{1,2}(K_0)\hookrightarrow L^2(K_0)$$
is compact, and since $\mathcal{W}_0$ is $L^\infty$ and has (compact) support included in $K_0$, the composition

$$\mathcal{H}^{-1/2}\mathcal{W}_0:L^2\rightarrow L^2(K_0)$$
is compact.

\end{proof}
We will also need the following result, which follows from arguments used in \cite[Definition 6]{D2}, that remain valid even if $M$ does not satisfy a global Sobolev inequality:

\begin{Lem}\label{L4c}

Assume that $M$ is non-parabolic. Then the following are equivalent:

\begin{enumerate}

\item[i)] There exists $\eta>0$ such that

$$||\mathcal{H}^{-1/2}\mathcal{W}_0\mathcal{H}^{-1/2}||_{2\to 2}\leq 1-\eta.$$

\item[ii)] $\mathcal{L}$ is strongly subcritical.

\item[iii)]  
$\mathrm{Ker}_{H_0^1}(\mathcal{L})=\{0\}.$
 
\end{enumerate}

\end{Lem}

\begin{proof}

Notice that by Lemma \ref{L4b}, the operator 

$$\mathcal{H}^{-1/2}\mathcal{W}_0\mathcal{H}^{-1/2}=(\mathcal{H}^{-1/2}\mathcal{W}_0^{1/2})(\mathcal{W}_0\mathcal{H}^{-1/2})$$
is self-adjoint, compact on $L^2$. Then, {\em mutatis mutandis}, the proof of \cite[Definition 6]{D2} (see also \cite[Remark 2]{D2}) applies, and shows the equivalence of i) and iii). For the sake of completeness, let us detail a bit more: for $u\in L^2$, define $\omega=\mathcal{H}^{-1/2}u\in H_0^1$. Then

%


$$\begin{array}{rcl}
(\mathcal{H}^{-1/2}\mathcal{W}_0\mathcal{H}^{-1/2}u,u)\leq (u,u)&\Leftrightarrow& (\mathcal{H}^{-1/2}\mathcal{W}_0\,\omega,\mathcal{H}^{1/2}\omega)\leq (\mathcal{H}^{1/2}\omega,\mathcal{H}^{1/2}\omega)\\
&\Leftrightarrow& (\mathcal{W}_0\,\omega,\omega)\leq (\mathcal{H}\omega,\omega).
\end{array}$$
The latter inequality is true since $\mathcal{L}=\mathcal{H}-\mathcal{W}_0$ is non-negative, hence 

$$(\mathcal{H}^{-1/2}\mathcal{W}_0\mathcal{H}^{-1/2}u,u)\leq ||u||_2,\qquad u\in L^2.$$
Since $\mathcal{H}^{-1/2}\mathcal{W}_0\mathcal{H}^{-1/2}$ is self-adjoint, this yields $||\mathcal{H}^{-1/2}\mathcal{W}_0\mathcal{H}^{-1/2}||_{2\to 2}\leq 1$. Since $\mathcal{H}^{-1/2}\mathcal{W}_0\mathcal{H}^{-1/2}$ is compact on $L^2$, either $1$ is eigenvalue of $\mathcal{H}^{-1/2}\mathcal{W}_0\mathcal{H}^{-1/2}$, or i) is satisfied for some $\eta>0$. But using that
$$\mathcal{L}=\mathcal{H}^{1/2}(I-\mathcal{H}^{-1/2}\mathcal{W}_0\mathcal{H}^{-1/2})\mathcal{H}^{1/2},$$
it can be shown that $\mathcal{H}^{-1/2}$ is an isometry from $\mathrm{Ker}_{L^2}(I-\mathcal{H}^{-1/2}\mathcal{W}_0\mathcal{H}^{-1/2})$ to $\mathrm{Ker}_{H_0^1}(\mathcal{L})$. Thus, $1$ is an $L^2$ eigenvalue of $\mathcal{H}^{-1/2}\mathcal{W}_0\mathcal{H}^{-1/2}$, if and only if $\mathrm{Ker}_{H_0^1}(\mathcal{L})\neq\{0\}$, and this proves that i) and iii) are equivalent.

Let us now prove the equivalence of i) and ii): first, let us assume ii), that is for some $\varepsilon>0$,

$$\langle \mathcal{R}_-\omega,\omega\rangle\leq (1-\varepsilon)\langle (\nabla^*\nabla +\mathcal{R}_+)\omega,\omega\rangle,\qquad \forall \omega\in C_0^\infty(E).$$
Since

$$\begin{array}{rcl}
\langle \mathcal{R}_-\omega,\omega\rangle&=&\langle \mathcal{W}_0\,\omega,\omega\rangle+\langle\mathcal{W}_\infty\,\omega,\omega\rangle\\\\
&\geq&\langle \mathcal{W}_0\,\omega,\omega\rangle,
\end{array}$$
one gets

$$\langle \mathcal{W}_0\,\omega,\omega\rangle\leq (1-\varepsilon)\langle (\nabla^*\nabla +\mathcal{R}_+)\omega,\omega\rangle,\qquad \forall \omega\in C_0^\infty(E).$$
As before, this implies that

$$\langle \mathcal{H}^{-1/2}\mathcal{W}_0\mathcal{H}^{-1/2}u,u\rangle\leq (1-\varepsilon)||u||_2,\qquad \forall u\in L^2$$
and since $\mathcal{H}^{-1/2}\mathcal{W}_0\mathcal{H}^{-1/2}$ is self-adjoint on $L^2$,

$$
||\mathcal{H}^{-1/2}\mathcal{W}_0\mathcal{H}^{-1/2}||_{2\to 2}\leq 1-\varepsilon,
$$
that is, i) with $\eta=\varepsilon$. Conversely, assume that i) holds, then by the same argument, we obtain
$$
\langle \mathcal{W}_0\,\omega,\omega\rangle\leq (1-\eta)\langle (\nabla^*\nabla +\mathcal{R}_+-\mathcal{W}_{\infty})\omega,\omega\rangle,\qquad \forall \omega\in C_0^\infty(E).$$
Therefore
$$
\langle \mathcal{R}_-\omega,\omega\rangle=\langle \mathcal{W}_0\,\omega,\omega\rangle+\langle\mathcal{W}_\infty\,\omega,\omega\rangle\leq (1-\eta)\langle (\nabla^*\nabla +\mathcal{R}_+)\omega,\omega\rangle+\eta \langle -\mathcal{W}_{\infty})\omega,\omega\rangle $$
for all $\omega\in C_0^\infty(E)$. By \eqref{eps-pos2}, there is $\varepsilon>0$ such that

$$\langle\mathcal{W}_\infty\,\omega,\omega\rangle\leq (1-\varepsilon) \langle \nabla^*\nabla\omega,\omega\rangle,$$
so

$$\langle \mathcal{R}_-\omega,\omega\rangle\leq (1-\eta+(1-\varepsilon)\eta)\langle (\nabla^*\nabla +\mathcal{R}_+)\omega,\omega\rangle.$$
Setting
$$(1-\eta+(1-\varepsilon)\eta)=1-\alpha$$
with $\alpha>0$, one obtains that $\mathcal{L}$ is $(1-\alpha)$-strongly subcritical. This concludes the proof that $i)$ is equivalent to ii).

\end{proof}
We will actually need the following slight variation of Lemma \ref{L4c}:

\begin{Lem}\label{L4}
Assume that $M$ is non-parabolic. For every $\lambda\geq0$, the operator $\mathcal{A}_\lambda$ is self-adjoint and compact on $L^2$. Furthermore, the following are equivalent:

\begin{enumerate}

\item[i)] There is $\eta\in (0,1)$ such that for all $\lambda\geq0$,
$$\|\mathcal{A}_\lambda\|_{2\to 2}\leq 1-\eta.$$

\item[ii)] $\mathcal{L}$ is strongly subcritical.

\item[iii)] $\mathrm{Ker}_{H_0^1}(\mathcal{L})=\{0\}.$

\end{enumerate}
\end{Lem}

\noindent{\em Proof: }  Denote $\mathcal{U}_\lambda=(\mathcal{H}+\lambda)^{-1/2}\mathcal{W}_0^{1/2}$. Then by Lemma \ref{L4b}, $\mathcal{U}_\lambda^*=\mathcal{W}_0^{1/2}(\mathcal{H}+\lambda)^{-1/2}$, and noticing that

\begin{equation*}\label{au}
\mathcal{A}_\lambda=\mathcal{U}_\lambda^*\mathcal{U}_\lambda,
\end{equation*}
one concludes by Lemma \ref{L4b} that $\mathcal{A}_\lambda$ is self-adjoint, compact on $L^2$.
Let us now observe that, since $\|\mathcal{U}_\lambda^* \mathcal{U}_\lambda\|_{2\to 2}=\|\mathcal{U}_\lambda^*\|^2_{2\to 2}$ , for any $\varepsilon>0$,
the following two conditions are equivalent:
$$\|\mathcal{A}_\lambda\|_{2\to 2}\le (1-\varepsilon)^2, \ \forall\,\lambda\geq 0,$$
and
$$\|\mathcal{U}_\lambda^*\|_{2\to 2}\le 1-\varepsilon, \ \forall\,\lambda\geq 0.$$
Remarking that $\mathcal{U}_\lambda^*=\mathcal{U}_0\mathcal{H}^{1/2}(\mathcal{H}+\lambda)^{-1/2}$, 
and that $\mathcal{H}^{1/2}(\mathcal{H}+\lambda)^{-1/2}$ is contractive on $L^2$ by the spectral theorem,  it follows that the latter condition is equivalent to
\begin{equation*}\|\mathcal{U}_0\|_{2\to 2}\le 1-\varepsilon,
\end{equation*}
that is
\begin{equation}\|\tilde{\mathcal{A}}_0\|_{2\to 2}\le (1-\varepsilon)^2,\label{A_0}\end{equation}
where $$
\tilde{\mathcal{A}}_0=\mathcal{U}_0\,\mathcal{U}_0^*=\mathcal{H}^{-1/2}\mathcal{W}_0\mathcal{H}^{-1/2}.
$$
Thus, i) is equivalent to \eqref{A_0}. By Lemma \ref{L4c}, \eqref{A_0} holds for some $\varepsilon>0$ if and only if ii) or iii) hold, which 
proves the claim.



\cqfd

Let us now recall some  notation  about the spectrum of a bounded linear operator. If $T$ is a bounded operator on $L^p$, $p\in[1,\infty]$, denote by $\sigma_p(T)$  the spectrum  of $T$. Also, denote by 

$$r_p(T)=\max\{|\mu|\,,\,\mu\in \sigma_p(T)\},$$
the spectral radius of $T$. We recall the well-known property that

$$\lim_{n\to +\infty}\|T^n\|_{p\to p}^{1/n}=r_p(T).$$

\begin{Lem}\label{L5}
Assume that
$$\|\mathcal{A}_\lambda\|_{2\to 2}\leq 1-\eta, \quad \forall \lambda\geq 0,$$
for some $\eta\in(0,1)$, then 
$$r_\infty(\mathcal{B}_\lambda)\leq 1-\eta$$
for all $\lambda\geq 0$.
\end{Lem}

\noindent{\em Proof: }  Let $\mu\in\sigma_\infty(\mathcal{B}_\lambda)\setminus\{0\}$. Recall that by Lemma  \ref{L1}, $\mathcal{B}_\lambda$ is compact on $L^\infty$; by Fredholm theory for compact operators on Banach spaces (see \cite{HP}), there exists a non-zero function  $\varphi\in L^\infty$  such that $\mathcal{B}_\lambda\varphi=\mu\varphi$, that is:

\begin{equation}\label{eigen}
(\mathcal{H}+\lambda)^{-1}\mathcal{W}_0\varphi=\mu\varphi.
\end{equation}
Applying $\mathcal{W}_0^{1/2}$ to both sides of \eqref{eigen} yields
$$\mathcal{W}_0^{1/2}(\mathcal{H}+\lambda)^{-1}\mathcal{W}_0\varphi=\mu \mathcal{W}_0^{1/2}\varphi,$$
that is
$$\mathcal{A}_\lambda \mathcal{W}_0^{1/2}\varphi=\mu \mathcal{W}_0^{1/2}\varphi.$$
Notice that $\mathcal{W}_0^{1/2}\varphi\not\equiv0$, since otherwise one would get $\mathcal{B}_\lambda\varphi=0$, that is $\mu=0$, which is a contradiction. Also, since $\varphi\in L^\infty$, using the fact that $\mathcal{W}_0$ has compact support, one sees that $\mathcal{W}_0^{1/2}\varphi\in L^2$. As a consequence,

$$\mu\in \sigma_2(\mathcal{A}_\lambda).$$
By Lemma \ref{L4}, there is $\eta>0$, independent of $\lambda$, such that
$$\|\mathcal{A}_\lambda\|_{2\to 2}\leq 1-\eta.$$
Therefore,

$$|\mu|\leq 1-\eta.$$
Since this is true for every $\mu\in \sigma_\infty(\mathcal{B}_\lambda)\setminus\{0\}$, one gets, by definition of the spectral radius, that $r_\infty(\mathcal{B}_\lambda)\leq 1-\eta$.

\cqfd

\bigskip

We are now ready to give the proof of  Proposition \ref{key}.

\bigskip

\noindent{\em Proof of Proposition $\ref{key}$:} 
Let $M$  satisfy the assumptions. By Lemma \ref{L5}, for every $\lambda\geq0$, the spectral radius of $\mathcal{B}_\lambda$ on $L^\infty$ is less than $1-\eta$. In particular, $1\not\in \sigma_\infty(\mathcal{B}_\lambda)$, and thus $(I-\mathcal{B}_\lambda)^{-1}$ is a bounded operator on $L^\infty$. It remains to show that
$$\sup_{\lambda\geq 0}\|(I-\mathcal{B}_\lambda)^{-1}\|_{\infty\to \infty}<+\infty.$$
First, recall that by (an easy case of) Lemma \ref{H_op},
$$\|(\mathcal{H}+\lambda)^{-1}\|_{\infty\to \infty}\leq \frac{C}{\lambda}.$$
On the other hand, $\mathcal{W}_0$ being compactly supported acts on all $L^\infty$ spaces.
Thus there exists $\Lambda>0$ such that for every $\lambda\geq\Lambda$,
\begin{equation}\label{largelambda}
\|\mathcal{B}_{\lambda}\|_{\infty\to \infty}=\|(\mathcal{H}+\lambda)^{-1}\mathcal{W}_0\|_{\infty\to \infty}\leq \frac{1}{2}.
\end{equation}
It follows that for every $\lambda\geq\Lambda$,
$$\|(I-\mathcal{B}_\lambda)^{-1}\|_{\infty\to \infty}\leq 2.$$
\bigskip
It remains to show that $\|(I-\mathcal{B}_\lambda)^{-1}\|_{\infty\to \infty}$ is bounded for $\lambda\in [0,\Lambda]$. For this, it is enough to prove that $\lambda\mapsto (I-\mathcal{B}_\lambda)^{-1}\in \mathcal{L}(L^\infty,L^\infty)$ is continuous on $[0,\infty)$. 

Let $\lambda_0\geq 0$. Since, by Lemma \ref{L5},
$$\lim_{n\to +\infty}\|\mathcal{B}_{\lambda_0}^n\|_{\infty\to \infty}^{1/n}\leq 1-\eta,$$
there exists $N$ such that $\|\mathcal{B}_{\lambda_0}^N\|_{\infty\to \infty}^{1/N}\leq 1-\frac{\eta}{2}$, hence for some $\alpha\in (0,1)$,

$$\|\mathcal{B}_{\lambda_0}^N\|_{\infty\to \infty}\leq 1-\alpha.$$
Now we know by Lemma \ref{L1} that the map $\lambda\mapsto \mathcal{B}_\lambda\in \mathcal{L}(L^\infty,L^\infty)$ is continuous on $[0,\infty)$.
Therefore, for $\lambda$ close enough to $\lambda_0$, one has
$$\|\mathcal{B}_{\lambda}^N\|_{\infty\to \infty}\leq 1-\frac{\alpha}{2}.$$
By Lemma \ref{L1} again, there is a constant $C\geq 1$ such that for every $\lambda\geq 0$,
$$\|\mathcal{B}_\lambda\|_{\infty\to \infty}\leq C.$$
Thus, for every $n\geq N$ and $\lambda$ close  to $\lambda_0$,
$$\|\mathcal{B}_\lambda^n\|_{\infty\to \infty}\leq C^N \left(1-\frac{\alpha}{2}\right)^{[n/N]}.$$
Thus, the series $\sum_{n\geq0}\|\mathcal{B}_\lambda^n\|_{\infty\to \infty}$ converges uniformly for $\lambda$ close  to $\lambda_0$, and this implies the continuity of $\lambda\mapsto (I-\mathcal{B}_\lambda)^{-1}$ at $\lambda_0$. The proof of Proposition \ref{key} is complete.

\cqfd

\section{Proofs of Theorem \ref{main}, Remark \ref{subc}, and Corollary \ref{scalar}}

\noindent{\em Proof of Theorem $\ref{main}$: }
 Observe first that under the assumptions of the theorem $M$ is non-parabolic, therefore Lemma \ref{L4} applies and ii) is equivalent to  $\mathrm{Ker}_{H_0^1}(\mathcal{L})=\{0\}$.
 
    Let $\mathcal{H}$ be defined by \eqref{defh}. By Lemma \ref{H_op},
\begin{equation}\label{resoh}
\sup_{t>0}\|(I+t\mathcal{H})^{-1}V_{\sqrt{t}}^{1/p}\|_{p\to\infty}<\infty,
\end{equation}
for all $p>\nu/2$. 
Since
$\mathcal{L}=\mathcal{H}-\mathcal{W}_0$, one may write the perturbation formula
\begin{equation}\label{pertu}
(1+t\mathcal{L})^{-1}=(I-(1+t\mathcal{H})^{-1}t\mathcal{W}_0)^{-1}(1+t\mathcal{H})^{-1}.
\end{equation}
Assume now ii), hence $\mathrm{Ker}_{H_0^1}(\mathcal{L})=\{0\}$.
Thus Proposition \ref{key} applies   and, taking $\lambda=1/t$, yields
\begin{equation}
\label{resol}\sup_{t>0}\|(I-(1+t\mathcal{H})^{-1}t\mathcal{W}_0)^{-1}\|_{\infty\to\infty}<+\infty.
\end{equation}
Gathering \eqref{resoh}, \eqref{pertu}, and \eqref{resol}, one obtains
$$\sup_{t>0}\|(I+t\mathcal{L})^{-1}V_{\sqrt{t}}^{1/p}\|_{p\to\infty}<\infty,$$
for all $p>\nu/2$. By Theorem \ref{blabla}, this implies $(U\!E_{\mathcal{L}})$. We have proved that ii) implies i). 

By \eqref{dd}, $M$ satisfies $(V^p$) for all $p\in \left(\frac{\kappa}{\kappa-2},+\infty\right)$, and by Lemma \ref{L7}, for all such $p$,

$$\mathrm{Ker}_{H_0^1}(\mathcal{L})\subset \mathrm{Ker}_{L^p}(\mathcal{L}).$$
Consequently, iii) implies that $\mathrm{Ker}_{H_0^1}(\mathcal{L})=\{0\}$, i.e. ii). Thus, we have proved so far that iii)$\Rightarrow$ ii) $\Rightarrow$ i). By Lemma \ref{converse}, i) implies iii), so i), ii) and iii) are equivalent. 

If $\kappa>4$, then by \eqref{dd}, ($V^2$) is satisfied, and Lemma  \ref{L7} implies that

$$\mathrm{Ker}_{L^2}(\mathcal{L})=\mathrm{Ker}_{H_0^1}(\mathcal{L}),$$
thus $\mathcal{L}$ is strongly subcritical if and only if $\mathrm{Ker}_{L^2}(\mathcal{L})=\{0\}$. 

\cqfd

{\em Proof of Remark $\ref{subc}$:} Recall that by non-parabolicity of $M$, $(V^\infty)$ is satisfied. By Lemma \ref{L4}, $\mathcal{L}$ strongly subcritical implies that 

$$\mathrm{Ker}_{H_0^1}(\mathcal{L})=\{0\}.$$
Thus, the conclusion of Proposition \ref{key} holds,  the proof of the implication ii)$\Rightarrow$ i) of Theorem \ref{main} applies, and shows that \eqref{UEL} holds.

\cqfd
{\em Proof of Corollary $\ref{scalar}$:} First, we already observed that if $M$ is non-parabolic, then the strong subcriticality of $L$ implies its  subcriticality. Thus, by \eqref{eps-pos}, the completion of $C_0^{\infty}(M)$ under the norm $\langle L^{1/2}u,L^{1/2}u\rangle$ injects continuously into $L^2_{loc}$, which implies that $L$ is subcritical. This shows that ii) implies iii).

Furthermore, by Remark \ref{subc}, if $L$ is strongly subcritical, then the heat kernel of $L$ has Gaussian estimates. Thus, ii) implies i).

Next, let us show that ii) and iii) are equivalent. Let us introduce the notations 

$$W_0(x):=\mathbf{1}_{K_0}(x)V_-(x),$$
and

$$W_\infty(x):=\mathbf{1}_{M\setminus K_0}(x)V_-(x).$$
Assume that $L$ is subcritical, and let us prove by contradiction that $V_-$ is strongly subcritical. If it is not the case, then by Lemma \ref{L4} there exists $\varphi\in \mathrm{Ker}_{H_0^1}(L)\setminus\{0\}$. By definition of $H_0^1$, there is a sequence of smooth compactly supported functions $(\varphi_n)_{n\in\mathbb{N}}$ such that

$$\varphi_n\rightarrow\varphi\mbox{ in }H_0^1.$$
Notice that since $L$ is non-negative, for every $u\in C_0^\infty(M)$,

$$\int_MW_0u^2\leq Q_H(u),$$
where we recall that $Q_H$ is the quadratic form of $H=\Delta+V_+-W_\infty$. Thus, $\int_MW_0u^2\leq ||u||_{H_0^1}^2<\infty$, for every $u\in H_0^1$, so the convergence of $(\varphi_n)_{n\in\mathbb{N}}$ to $\varphi$ in $H_0^1$ implies that

$$\lim_{n\to\infty}\int_MW_0\varphi_n=\int_MW_0\varphi.$$
Consequently,

$$
\lim_{n\to\infty}\left(Q_H(\varphi_n)-\int_MW_0\varphi_n^2\right)=Q_H(\varphi)-\int_MW_0\varphi^2.$$
But since $\varphi\in \mathrm{Ker}_{H_0^1}(L)$, one obtains by integration by parts that for every $\psi\in C_0^\infty(M)$,

$$Q_H(\varphi,\psi)-\int_MW_0\varphi\psi=0,$$
and by density of $C_0^\infty(M)$ in $H_0^1$, we conclude that

$$Q_H(\varphi)-\int_MW_0\varphi^2=0.$$
Thus, 

$$
\lim_{n\to\infty}\left(Q_H(\varphi_n)-\int_MW_0\varphi_n^2\right)=0.
$$
Furthermore, by \eqref{vNP}, $(\varphi_n)_{n\in\mathbb{N}}$ converges to $\varphi$ in $L^2_{loc}$, so $\varphi$ is a {\em null-state} (see \cite[Definition 1.1]{PT}) of $L$. By \cite[Theorem 1.4]{PT}, this implies that $L$ is critical, which is a contradiction. This shows that iii) implies ii).

It remains to show that i) implies iii). Assume i), then for all $x\neq y$,

$$\int_1^\infty |e^{-tL}(x,y)|\,dt\leq C\int_1^\infty \frac{dt}{V(x,\sqrt{t})}.$$
Since $M$ is non-parabolic and satisfies \eqref{d} and \eqref{UE}, ($V^\infty$) holds, so the above integral is finite. This implies that $L$ is subcritical. So, iii) holds, which concludes the proof.

\cqfd

\section{Proofs of Theorem \ref{thm_para}, Corollary \ref{4cons} and Corollary \ref{Riesz_hodge}}

\noindent {\em Proof of Theorem $\ref{thm_para}$:} In the case $\nu<2$, one can prove directly --that is, without iterating-- that 

\begin{equation}\label{resol1}
\sup_{t>0}||V_{\sqrt{t}}^{1/2}(I+t\mathcal{L})^{-1}||_{2\to\infty}<\infty.
\end{equation}
We define $\mathcal{H}$ in a way that is different from the one used in the proof of Theorem \ref{main}: let us denote $\mathcal{H}=\bar{\Delta}+\mathcal{R}_+$. We pass to $\lambda=1/t$, and write

$$(\mathcal{L}+\lambda)^{-1}=(\mathcal{H}+\lambda)^{-1/2}(I-(\mathcal{H}+\lambda)^{-1/2}\mathcal{R}_-(\mathcal{H}+\lambda)^{-1/2})^{-1}(\mathcal{H}+\lambda)^{-1/2}.$$
Denote 

$$\tilde{\mathcal{A}}_\lambda:=(\mathcal{H}+\lambda)^{-1/2}\mathcal{R}_-
(\mathcal{H}+\lambda)^{-1/2},$$
and assume first that $\mathcal{R}_-$ is in $L^\infty$. Since

$$||(\mathcal{H}+\lambda)^{-1/2}||_{2\to2}\leq \lambda^{-1/2},$$
for $\lambda>0$ the operator $\tilde{A}_\lambda$ is self-adjoint, bounded on $L^2$. Then, by the proof of Lemma \ref{L4c}, the strong subcriticality assumption on $\mathcal{L}$ implies that there exists $\varepsilon>0$ such that for all $\lambda>0$,

$$||\tilde{\mathcal{A}}_\lambda||_{2\to2}\leq 1-\varepsilon.$$
We emphasize that here, contrary to Lemma \ref{L4}, one has to restrict ourselves to $\lambda>0$ in order that $\tilde{\mathcal{A}}_\lambda$ acts on $L^2$, because $\mathcal{R}_-$ is not in general compactly supported.
If $\mathcal{R}_-$ is only in $L^\infty_{loc}$ but not in $L^\infty$, then by an  approximation argument, one proves that $\tilde{\mathcal{A}}_\lambda$ is still self-adjoint on $L^2$ and has norm bounded by $1-\varepsilon$.

Therefore, the Neumann series $(I-\tilde{A}_\lambda)^{-1}=\sum_{n\geq 0}\tilde{A}_\lambda^n$ converges in  operator norm on $L^2$, and as a consequence

$$\sup_{\lambda>0}||(I-\tilde{A}_\lambda)^{-1}||_{2\to2}<\infty.$$
Hence,
\begin{eqnarray*}
&&||V_{\sqrt{t}}^{1/2}(I+t\mathcal{L})^{-1}||_{2\to\infty}\\
&\leq& ||V_{\sqrt{t}}^{1/2}(I+t\mathcal{H})^{-1/2}||_{2\to\infty}||(I-\tilde{\mathcal{A}}_\lambda)^{-1}||_{2\to2}||(I+t\mathcal{H})^{-1/2}||_{2,2}\\
&\leq & C ||V_{\sqrt{t}}^{1/2}(I+t\mathcal{H})^{-1/2}||_{2\to\infty}
(I+t\mathcal{H})^{-1/2}||_{2,2}.
\end{eqnarray*}
By the spectral theorem, 

$$\sup_{t>0}||(I+t\mathcal{H})^{-1/2}||_{2,2}\leq 1.$$
In the case $\nu<2$, by Lemma \ref{H_op}, we get

$$\sup_{t>0}||V_{\sqrt{t}}^{1/2}(I+t\mathcal{H})^{-1/2}||_{2\to\infty}<\infty.$$
Thus, one obtains

$$\sup_{t>0}||V_{\sqrt{t}}^{1/2}(I+t\mathcal{L})^{-1}||_{2\to\infty}<\infty.$$
This implies \eqref{UEL}.

In the case $\nu=2$, by Lemma \ref{H_op}, one obtains that for all $p\in [2,\infty)$,

$$\sup_{t>0}||V_{\sqrt{t}}^{1/2}(I+t\mathcal{H})^{-1/2}||_{2\to p}<\infty.$$
By Theorem \ref{thm_VEV}, this implies $(V\!EV_{p,q,\gamma})$, for all $1<p\leq q<\infty$ and all $\gamma\in \R$.

Therefore, \eqref{resol1} is proved.

\cqfd




\noindent{\em Proof of Corollary $\ref{Riesz_hodge}$:} By Corollary \ref{hodge}, the heat kernel of $\Delta_k$ has Gaussian estimates. By integration by parts,  $(d+d^*)\Delta_k^{-1/2}$ is an isometry on $L^2$, and the first claim now follows from \cite[Theorem 5]{Sik}.

If moreover $\mathcal{R}_{k+1}$ satisfies \eqref{condik} and  $H_2^{k+1}(M)=\{0\}$ for some $p\in (\frac{\kappa}{\kappa-2},\infty)$, then by Corollary \eqref{hodge}, the heat kernel of $\Delta_{k+1}$ has Gaussian estimates. Thus, by the same argument $d^*_{k+1}\Delta_{k+1}^{-1/2}$ is bounded from $L^p(\Lambda^{k+1}T^*M)$ to $L^p(\Lambda^kT^*M)$, for all $p\in (1,2]$. Taking the adjoint, we deduce that $\Delta_{k+1}^{-1/2}d_{k}$ is bounded from $L^p(\Lambda^kT^*M)$ to $L^p(\Lambda^{k+1}T^*M)$. Using the fact that the Hodge Laplacian and $d$, commute, we deduce that $d_k\Delta^{-1/2}_{k}$ is bounded on $L^p$ for all $p\in (1,\infty)$. The proof for $d_{k-1}^*\Delta_k^{-1/2}$ is the same, starting from the boundedness of $d_{k-1}\Delta^{-1/2}_{k-1}$ on $L^p$, $p\in (1,2)$ and taking the adjoint.

\cqfd

\section{Proof of Theorem \ref{Riesz}}

We follow the strategy developed originally in \cite{D2}: the proof is by a perturbation argument, using the Gaussian estimates result from Theorem~\ref{main}. By hypothesis, $\mathrm{Ric}_-$ satisfies condition \eqref{condik}, so by Theorem \ref{main} we know that there is $\mathcal{W}$ is a non-negative, smooth, {\em compactly supported} potential such that $\vec{\Delta}+\mathcal{W}$ is strongly subcritical, and the heat kernel $e^{-t(\vec{\Delta}+\mathcal{W})}$ has Gaussian estimates. Here, we have identified the potential $\mathcal{W}$ with the operator $\mathcal{W}$ times the identity acting on $1$-forms. More precisely, $\mathcal{W}$ is chosen so that

$$\sup_{x\in M}\int_M G(x,y)||(\mathrm{Ric}+\mathcal{W})_-(y)||_y\,dy<1.$$ 
Furthermore, since $\mathcal{W}$ is non-negative,

$$\int_M|d\omega|^2+|d^*\omega|^2\leq  \langle (\vec{\Delta}+\mathcal{W})\omega,\omega\rangle,\qquad\forall \omega\in C_0^\infty(E).$$
Thus, $d^*(\vec{\Delta}+\mathcal{W})^{-1/2}$ is bounded on $L^2$.  Since $e^{-t(\vec{\Delta}+\mathcal{W})}$ has Gaussian estimates, by \cite{Sik} one deduces that $d^*(\vec{\Delta}+\mathcal{W})^{-1/2}$ is bounded on $L^q$ for all $q\in (1,2)$. Taking the adjoint, one gets that $(\vec{\Delta}+\mathcal{W})^{-1/2}d$ is bounded on $L^q$, for all $q\in [2,\infty)$. However, due to the presence of the potential $\mathcal{W}$, one cannot commute the operators $d$ and $(\vec{\Delta}+\mathcal{W})^{-1/2}$. In what follows, we fix $3<p<\kappa$. Denote $K_0$ the (compact) support of $\mathcal{W}$. The proof of Theorem \ref{Riesz} relies on the following two claims:\\

{\em Claim 1:} $d(\Delta+\mathcal{W})^{-1/2}$ is bounded on $L^p$.\\

{\em Claim 2:} $d\Delta^{-1/2}-d(\Delta+\mathcal{W})^{-1/2}$ is bounded on $L^p$.\\

Indeed, assuming for the moment the validity of these two claims, we conclude that $d\Delta^{-1/2}=(d\Delta^{-1/2}-d
(\Delta+\mathcal{W})^{-1/2})+d(\Delta+\mathcal{W})^{-1/2}$ is bounded on $L^p$. Since $M$ satisfies \eqref{d} and \eqref{UE}, \cite{CD1} implies that $d\Delta^{-1/2}$ is also bounded on $L^q$ for all $q\in (1,2)$. Since $3<p<\kappa$ is arbitrary, one concludes by interpolation that $d\Delta^{-1/2}$ is bounded on $L^p$ for all $p\in (1,\kappa)$. This finishes the proof of Theorem \ref{Riesz}.\\

{\em Proof of Claim 1:} Given the volume lower estimate \eqref{dd}, the proof is a simple adaptation of the proof of \cite[Theorem 16]{D2}. It relies on ideas developed in \cite{C}. Let us explain briefly the main steps of the proof. 
The idea is to show that one can commute the operators $d$ and $(\vec{\Delta}+\mathcal{W})^{-1/2}$, modulo an error which is bounded in $L^p$. First, the Gaussian estimates satisfied by $e^{-t(\vec{\Delta}+\mathcal{W})}$ and \eqref{dd} imply that for all $1\leq r\leq s\leq\infty$,

\begin{equation}\label{heat1}
||e^{-t(\vec{\Delta}+\mathcal{W})}||_{L^r(K_0)\to L^s}\lesssim t^{-\frac{\kappa}{2}\left(\frac{1}{r}-\frac{1}{s}\right)},\qquad \forall t\geq1.
\end{equation}
By subordination, \eqref{heat1} implies

\begin{equation}\label{heat2}
||e^{-t\sqrt{\vec{\Delta}+\mathcal{W}}}||_{L^r(K_0)\to L^s}\lesssim t^{-\kappa\left(\frac{1}{r}-\frac{1}{s}\right)},\qquad \forall t\geq1.
\end{equation}
Notice that since $\mathcal{W}\geq0$, a similar estimate holds for the Poisson semi-group of $\Delta+\mathcal{W}$ (that is, on functions); we will use it in the following form:

\begin{equation}\label{heat3}
||e^{-t\sqrt{\Delta+\mathcal{W}}}||_{L^p\to L^\infty(K_0)}\lesssim t^{-\kappa/p},\qquad \forall t\geq1.
\end{equation}
Furthermore, for all $1\leq q\leq \infty$,

\begin{equation}\label{heat4}
||e^{-t\sqrt{\Delta+\mathcal{W}}}||_{L^q\to L^q}\leq 1,\qquad \forall t>0.
\end{equation}
(this follows from $e^{-(\Delta+\mathcal{W})}\leq e^{-t\Delta}$, and subordination). Consider the vector bundle $E$ over $\R_+\times M$, whose fiber at $(t,p)$ is $\Lambda^1T^*M$. Acting on sections of $E$, consider the elliptic operator (the wave operator) $\mathcal{P}=-\frac{\partial^2}{\partial t^2}+\vec{\Delta}+\mathcal{W}$, with Neumann boundary conditions at the boundary $t=0$. Let $\mathcal{G}$ be the Green function of $\mathcal{P}$, whose kernel is given explicitly by

\begin{equation}\label{green}
\mathcal{G}(\sigma,s,x,y)=\int_0^\infty \left(\frac{e^{-\frac{(\sigma-s)^2}{4t}}-e^{-\frac{(\sigma+s)^2}{4t}}}{\sqrt{4\pi t}}\right)e^{-t(\vec{\Delta}+\mathcal{W})}(x,y)\,dt.
\end{equation}
Then, for $u\in C_0^\infty(M)$, one has the approximate commutation relation:

\begin{equation}\label{commute}
e^{-\sigma\sqrt{\vec{\Delta}+\mathcal{W}}}du=de^{-\sigma\sqrt{\Delta+\mathcal{W}}}u-\mathcal{G}(\mathcal{P}(de^{-\sigma\sqrt{\Delta+\mathcal{W}}}u))
\end{equation}
(see \cite{D2} or \cite{COM} for more details). Denote 

$$E[u]=\mathcal{G}(\mathcal{P}(de^{-\sigma\sqrt{\Delta+\mathcal{W}}}u)).$$
Integrating \eqref{commute} from $t=0$ to $\infty$, one obtains

\begin{equation}\label{com_int}
(\vec{\Delta}+\mathcal{W})^{-1/2}du=d(\Delta+\mathcal{W})^{-1/2}u-\int_0^\infty E[u](\sigma)\,d\sigma.
\end{equation}
It is easy to compute that 

\begin{equation}\label{error_comp}
\mathcal{P}(de^{-\sigma\sqrt{\Delta+\mathcal{W}}}u)=-(e^{-\sigma\sqrt{\Delta+\mathcal{W}}}u)(d\mathcal{W}),
\end{equation}
which has (compact) support included in $K_0$. It follows, using the heat kernel estimates \eqref{heat1} and \eqref{heat2}, that

\begin{equation}\label{error_est}
\left|\left|\int_0^\infty E[u](\sigma)\,d\sigma\right|\right|_p\lesssim ||u||_p
\end{equation}
(see the computation below, where the assumption $\kappa>3$ is needed). Since we know that $(\vec{\Delta}+\mathcal{W})^{-1/2}du$ is bounded on $L^p$, from \eqref{com_int} \eqref{error_est} one deduces that $d(\Delta+\mathcal{W})^{-1/2}$ is bounded on $L^p$.\\

{\em Proof of Claim 2:} Note that the proof of \cite[Theorem 15]{D2} cannot be adapted to the present setting, but given the volume estimate \eqref{dd} and the fact that $\mathcal{W}$ has compact support, one can use instead ideas from \cite[Section 3.6]{AO}. The argument is presented in details in \cite[Proof of Theorem 4.1, Step III]{COM}.\\

{\em Proof of  estimate \eqref{error_est}:} 
First, by \eqref{error_comp} and \eqref{heat3}, \eqref{heat4}, and the fact that $\mathcal{W}$ has compact support, one has

\begin{equation}\label{er1}
||\mathcal{P}(de^{-\sigma\sqrt{\Delta+\mathcal{W}}}u)||_1+||\mathcal{P}(de^{-\sigma\sqrt{\Delta+\mathcal{W}}}u)||_p\lesssim (1+\sigma)^{-\kappa/p}||u||_p.
\end{equation}
The proof of \eqref{error_est} now follows the computations in \cite{C}; first, denote $f(t,x)=\mathcal{P}(de^{-t\sqrt{\Delta+\mathcal{W}}}u)(x)$, then

$$E[u](\sigma,x)=\int_{\R_+\times M}\mathcal{G}(\sigma,s,x,y)f(s,y)\,ds\,d\mu(y),$$
and 

$$\int_0^\infty E[u](\sigma,x)\,d\sigma=\int_{\R_+^2\times M}\mathcal{G}(\sigma,s,x,y)f(s,y)\,d\sigma ds\,d\mu(y).$$
Integrating \eqref{green}, one finds that

$$\begin{array}{rcl}
\int_0^\infty \mathcal{G}(\sigma,s,x,y)\,d\sigma&=&\frac{1}{4\pi}\int_0^\infty\left[\int_{-s}^se^{-\frac{v^2}{4t}}\,dv\right]e^{-t(\vec{\Delta}+\mathcal{W})}(x,y)\,\frac{dt}{\sqrt{t}}\\\\
&=&\frac{2}{\sqrt{\pi}}\int_0^\infty e^{-r^2}\left[\int_0^{\frac{s^2}{4r^2}}e^{-t(\vec{\Delta}+\mathcal{W})}\,dt\right]\,dr.
\end{array}$$
Therefore,

$$\int_0^\infty E[u](\sigma,x)\,d\sigma=\frac{2}{\sqrt{\pi}}\int_{\R^2_+}\left[\int_0^{\frac{s^2}{4r^2}}(e^{-t(\vec{\Delta}+\mathcal{W})}f(s,\cdot))(x)\,dt\right]\,drds.$$
Since $e^{-t(\vec{\Delta}+\mathcal{W})}$ has Gaussian estimates, it is uniformly bounded on $L^p$, and therefore

$$ ||e^{-t(\vec{\Delta}+\mathcal{W})}f(s,\cdot))||_p\lesssim ||f(s,\cdot)||_p\lesssim  (1+s)^{-\kappa/p}||u||_p, \qquad \forall t\leq 1.$$
Also, by the fact that $f(s,\cdot)$ has support included in $K_0$, \eqref{heat1} and \eqref{er1}, for all $t>1$,

$$\begin{array}{rcl}
||e^{-t(\vec{\Delta}+\mathcal{W})}f(s,\cdot))||_p&\lesssim& ||e^{-t(\vec{\Delta}+\mathcal{W})}||_{L^1(K_0)\to L^p}||f(s,\cdot)||_1\\\\
&\lesssim& t^{-\frac{\kappa}{2}(1-\frac{1}{p})}(1+s)^{-\kappa/p}||u||_p.
\end{array}$$
Hence,

$$\left|\left|\int_0^\infty E[u](\sigma,\cdot)\,d\sigma\right|\right|_p\lesssim \left(\int_{\R^2_+}\left[\int_0^{\frac{s^2}{4r^2}}\min(1,t^{-\frac{\kappa}{2}(1-\frac{1}{p})})(1+s)^{-\kappa/p}\,dt\right]\,drds\right)||u||_p.$$
The above integral is equal to

$$\int_{\{2r\sqrt{t}\leq s\}}e^{-r^2}\min(1,t^{-\frac{\kappa}{2}(1-\frac{1}{p})})(1+s)^{-\kappa/p}\,dt dr ds,$$
which, since $p<\kappa$, is easily seen to be equal to

$$\frac{\kappa}{\kappa-p}\int_{\R_+^2}e^{-r^2}\min(1,t^{-\frac{\kappa}{2}(1-\frac{1}{p})})(1+2r\sqrt{t})^{-\kappa/p-1}\,dr ds,$$
which is finite if and only if $p>\frac{\kappa}{\kappa-1}$ and $\kappa>3$. By assumption, these two conditions are satisfied, and thus

$$\left|\left|\int_0^\infty E[u](\sigma,\cdot)\,d\sigma\right|\right|_p\lesssim ||u||_p,$$
that is, \eqref{error_est}.

\cqfd

\section{Appendix: domination, Bochner formulas and Davies-Gaffney on weighted manifolds}
In this Appendix, we plan to explain:

\begin{enumerate}

\item  Why standard domination theory extends to the weighted case.

\item What is the Bochner formula for the weighted Hodge Laplacian. 

\item Why Gaussian off- and on-diagonal estimates for a generalised non-negative Schr\"{o}dinger operator $\mathcal{L}=\nabla^*\nabla+\mathcal{R}$ are equivalent.
 
\end{enumerate} 
 In all this section, $(M,\mu)$, $\mu=e^{f}dx$, will be a complete weighted manifold, and $E\to M$ be a vector bundle with a connection compatible with the metric. Recall the weighted Laplacian $\Delta_\mu=\Delta-( \nabla f,\nabla\cdot)$, and the weighted rough Laplacian $\bar{\Delta}_\mu=\nabla^*_\mu\nabla$ acting on sections of $E$. We first state a formula relating the weighted rough Laplacian to the unweighted one:

\begin{Lem}\label{rough_w}

The following formula hold, for $\xi$ smooth, compactly supported section of $E$:

$$\bar{\Delta}_\mu(\xi)=\bar{\Delta} (\xi) - \nabla_{\nabla f} \xi.$$

\end{Lem}

\begin{proof}

It is easy to see by integration by parts that $\nabla^*_{\mu}=e^{-f}\nabla^*e^{f}$. Let $p$ be a point of $M$, and $(e_1,\cdots,e_n)$ an orthonormal basis of $T_pM$. Notice that $\nabla \xi$ is a section of $T^*M\otimes \Lambda^kT^*M$, and that $\nabla^*$ acts on a section $\alpha\otimes \omega$ of $T^*M\otimes \Lambda^kT^*M$ at the point $p$ according to the formula:

$$\nabla^*\alpha\otimes \omega|_p=-\sum_{i=1}^n\alpha_p(e_i)\nabla_{e_i}\omega|_p.$$
So we see that

$$e^{-f}\nabla^*\alpha\otimes e^{f}\omega|_p=\nabla^*\alpha\otimes \omega|_p-\alpha(\nabla f)\omega|_p.$$
Writing that at $p$,

$$\nabla \xi=\sum_{i=1}^ne_i^*\otimes\nabla_{e_i}\xi=\sum_{i=1}^n\alpha_i\otimes\omega_i,$$
we see that

$$\sum_{i=1}^n\alpha_i(\nabla f)\omega_i|_p=\nabla_{\nabla f}\xi|_p.$$
Recalling that $\bar{\Delta}_\mu=e^{-f}\nabla^*e^{f}\nabla$, we get

$$\bar{\Delta}_\mu(\xi)=\bar{\Delta} (\xi) - \nabla_{\nabla f} \xi.$$

\end{proof}
Then the following extension of the results of \cite{HSU} holds:

\begin{Thm}

For any smooth, compactly supported section $\xi$ of $E$, and every $t\geq0$,

$$|e^{-t\bar{\Delta}_\mu}\xi|\leq e^{-t\Delta_\mu}|\xi|.$$

\end{Thm}

\begin{proof}
We follow G. Besson's proof of Theorem 20 in the appendix of \cite{Bes}. For $\varepsilon>0$, define 

$$|\xi|_\varepsilon=(|\xi|^2+\varepsilon^2)^{1/2}.$$
The only thing that is unclear is whether the formula

\begin{equation}\label{eq_B}
\Delta_\mu|\xi|_\varepsilon=\frac{( \bar{\Delta}_\mu\xi,\xi)}{|\xi|_\varepsilon}-\left(\frac{|\nabla \xi|^2}{|\xi|_\varepsilon}-\frac{( \nabla \xi,\xi)}{|\xi|_\varepsilon}\right),
\end{equation}
appearing on p. 173 of \cite{Bes} holds in the weighted case. If it is the case, then the whole proof of \cite[Theorem 20]{Bes} extends to the weighted case. Our starting point is formula \eqref{eq_B}, which holds for $\mu$ the Riemannian measure, by \cite{Bes}. Using Lemma \ref{rough_w}, we see that

$$(\bar{\Delta}_\mu \xi,\xi) = (\bar{\Delta} \xi,\xi) - (1/2) (\nabla f) \cdot (|\xi|^2),$$ 
and thus

$$\frac{(\bar{\Delta}_\mu \xi,\xi)}{|\xi|_\epsilon} = \frac{(\bar{\Delta} \xi,\xi)}{|\xi|_\epsilon} + (\nabla f)\cdot |\xi|_\epsilon.$$
On the other hand,

$$\begin{array}{rcl}
\Delta_\mu (|\xi|_\epsilon)&=&\Delta |\xi|_\epsilon - (\nabla f,\nabla |\xi|_\epsilon)\\\\
&=&\Delta |\xi|_\epsilon - (\nabla f)\cdot |\xi|_\epsilon.
\end{array}
$$
Thus, from the fact that \eqref{eq_B} holds for the Riemannian measure, we see that it also holds in the weighted case.

\end{proof}
We now proceed to the Bochner formula for the Hodge Laplacian. We first need to recall some facts about differential forms.  Given a vector field $X$ and a differential form $\alpha$ of degree $k$, the contraction of $X$ and $\alpha$ is $\iota_X\alpha$, a differential form of degree $k-1$ defined by

$$\iota_X\alpha(X_1,\cdots,X_{k-1})=\alpha(X,X_1,\cdots,X_{k-1}).$$
We will need the fact that the adjoint operation of $i_X$ is $X^\flat\wedge$, where $X^\flat$ is the $1$-form associated to $X$. Also of use for us will be the Cartan formula for the Lie derivative $\mathscr{L}_X\omega$ of a form $\omega$ with respect to a vector field $X$:

\begin{equation}\label{Cartan}
\mathscr{L}_X\omega=d(\iota_X\omega)+\iota_X(d\omega).
\end{equation}
Let us now consider the weighted Hodge Laplacian

$$\vec{\Delta}_\mu=dd^*_\mu+d^*_\mu d.$$
Here, $d^*_\mu$ is the adjoint of $d$, which of course depends on the measure $\mu$. For a discussion concerning the domain of $d^*_\mu$ and $\vec{\Delta}_\mu$, we refer the reader to \cite{Bue}. Let us just mention that since $M$ is complete, $\vec{\Delta}_\mu$ as defined above is equal to the closure of the essentially self-adjoint operator $d\delta+\delta d$, where $\delta$ is the formal adjoint of $d$.

\begin{Lem}\label{Hodge_w}

The following formulas hold:

$$d^*_\mu=d^*-\iota_{\nabla f},$$
and

$$\vec{\Delta}_\mu=\vec{\Delta}-\mathscr{L}_{\nabla f}.$$

\end{Lem}

\begin{proof}

By integration by parts, we see that

$$d^*_\mu=e^{-f}d^*e^{f}.$$
Using an easy integration by parts and that the adjoint of $\iota_X$ is $X^\flat\wedge$, it is not hard to see that for $u\in C_0^\infty(M)$ and $\xi\in \Lambda^kT^*M$,

$$d^*(u\xi)=ud^*\xi-\iota_{\nabla u}\xi.$$
This implies that

$$d^*_\mu=d^*-\iota_{\nabla f}.$$
Now, by the Cartan formula,

$$\begin{array}{rcl}
\vec{\Delta}_\mu&=&d(d^*-\iota_{\nabla f})+(d^*-\iota_{\nabla f})d\\\\
&=&dd^*+d^*d-d\iota_{\nabla f}-\iota_{\nabla f}d\\\\
&=&\vec{\Delta}-\mathscr{L}_{\nabla f},
\end{array}$$
which is the claimed result.

\end{proof}
For a smooth function $u$ on $M$, recall that the Hessian of $u$ is the symmetric $(0,2)$-tensor defined as

$$\mathrm{Hess}_u(X,Y)=\nabla_X\nabla_Yu-\nabla_{\nabla_XY}u.$$
Following \cite{EF}, we introduce $\mathscr{H}_u$, the {\em Hessian operator} of $u$, which is an operator acting on differential forms. Given a point $p$ of $M$, consider an orthonormal basis $(e_i)_{i=1}^n$ of $T_pM$, and let $(e_i^*)_{i=1}^n$ be the dual basis. Then, $\mathscr{H}_u$ is defined by

$$\mathscr{H}_u\omega=\sum_{i,j}\mathrm{Hess}_u(e_i,e_j)e_j^*\wedge\iota_{e_i}\omega.$$
We will need the following formula relating the Lie derivative and the covariant derivative on differential forms, and which is proved in \cite[Lemma 1.1]{EF}:

\begin{equation}\label{hess}
\mathscr{L}_{\nabla u}=\nabla_{\nabla u}-\mathscr{H}_u.
\end{equation}
We are now ready to prove:

\begin{Thm}\label{Bochner}

The following Bochner formula for the weighted Hodge Laplacian on $k$-forms holds:

$$\vec{\Delta}_{\mu}=\bar{\Delta}_\mu+\mathscr{R}_k-\mathscr{H}_f.$$

\end{Thm}

\begin{proof}
By Lemma \ref{Hodge_w} and the unweighted Bochner formula,

$$\vec{\Delta}_{k,\mu}=\bar{\Delta}+\mathscr{R}_k-\mathscr{L}_{\nabla f}.$$
By Lemma \ref{rough_w}, we get that

$$\vec{\Delta}_{k,\mu}=\bar{\Delta}_\mu+\mathscr{R}_k-\mathscr{L}_{\nabla f}+\nabla_{\nabla f}.$$
Then, the result follows from \eqref{hess}.

\end{proof}
Despite the fact that all the ingredients needed for the proof of Theorem \ref{Bochner} are quite classical, to the authors' knowledge, the result of Theorem \ref{Bochner} has never been stated explicitely in the literature, even if closely related formula are proved in \cite{Bue}. Let us explain now why the result of Theorem \ref{Bochner}, in the particular case $k=1$, implies easily the formula for the  iterated {\em carr\'{e} du champ} $\Gamma_2$ of the weighted Laplacian (see \cite[Proposition 3]{BE}). According to \cite{BE}, $\Gamma_2$ is defined by

$$\Gamma_2(u,v)=\frac{1}{2}\left(\Delta_\mu(\nabla u,\nabla v)-(\nabla \Delta_\mu u,\nabla v)-(\nabla u,\nabla \Delta_\mu v)\right).$$
Then, the following formula for the iterated  {\em carr\'{e} du champ}  of a weighted Laplacian is presented in \cite[Proposition 3]{BE}:

\begin{equation}\label{BE}
\Gamma_2(u,v)=-(\mathrm{Hess}_u,\mathrm{Hess}_v)-(\mathrm{Ric}-\mathrm{Hess}_u)(\nabla u,\nabla v).
\end{equation}
Notice that in \cite{BE}, the Laplacian is taken with the opposite sign convention, which changes the sign of $\Gamma_2$; that's why our formula \eqref{BE} has opposite sign.  We claim that \eqref{BE} follows quite easily from Theorem \ref{Bochner}: notice first that since the weighted Hodge Laplacian commutes with the differential $d$,

$$(\nabla \Delta_\mu u,\nabla v)=(d \Delta_\mu u,d v)=(\vec{\Delta}_\mu du,dv),$$
which, by Theorem \ref{Bochner} for $k=1$ and the fact that $\mathscr{R}_1$ identifies canonically with the Ricci curvature, gives

$$(\nabla \Delta_\mu u,\nabla v)=(\bar{\Delta}_\mu du,dv)+\mathrm{Ric}(du,dv)-(\mathscr{H}_fdu,dv).$$
It is easy to see from the definition of $\mathscr{H}_f$ that

$$(\mathscr{H}_fdu,dv)=\mathrm{Hess}_f(\nabla u,\nabla v).$$
Thus, using the symmetry of $\mathrm{Hess}_f$,

\begin{equation}\label{Eq-w}
\begin{array}{rcl}
\frac{1}{2}\left((\nabla \Delta_\mu u,\nabla v)+(\nabla u,\nabla \Delta_\mu v)\right)&=&\frac{1}{2}\left((\bar{\Delta}_\mu du,dv)+(du,\bar{\Delta}_\mu dv)\right)\\\\
&&+\mathrm{Ric}(du,dv)-\mathrm{Hess}_f(\nabla u,\nabla v).
\end{array}
\end{equation}
Next, let us compute $\Delta_\mu(\nabla u,\nabla v)$: by Lemma \ref{rough_w},

$$\begin{array}{rcl}
\Delta_\mu (\nabla u,\nabla v)&=&\Delta_\mu (d u,d v)\\\\
&=&-\mathrm{Tr}_{X,Y}\nabla_{X,Y}^2(du,dv)+(\nabla_{\nabla f}du,dv)+(du,\nabla_{\nabla f}dv)\\\\
&=&-\mathrm{Tr}_{X,Y}(\nabla_{X,Y}^2du,dv)-\mathrm{Tr}(du,\nabla_{X,Y}^2dv)\\\\
&&-2\mathrm{Tr}_{X,Y}(\nabla_X du,\nabla_Y dv)+(\nabla_{\nabla f}du,dv)+(du,\nabla_{\nabla f}dv)\\\\
&=&(\bar{\Delta}du,dv)+(du,\bar{\Delta}dv)-2\mathrm{Tr}_{X,Y}(\nabla_X du,\nabla_Y dv)\\\\
&&+(\nabla_{\nabla f}du,dv)+(du,\nabla_{\nabla f}dv)\\\\
&=&(\bar{\Delta}_{\mu} du,dv)+(du,\bar{\Delta}_{\mu} dv)-2\mathrm{Tr}_{X,Y}(\nabla_X du,\nabla_Y dv).
\end{array}$$
Using the fact that $\mathrm{Hess}_h$ can be computed by the formula:

$$\mathrm{Hess}_h(X,Y)=(X,\nabla_Y\nabla h),$$
one sees that

$$\mathrm{Tr}_{X,Y}(\nabla_X du,\nabla_Y dv)=(\mathrm{Hess}_u,\mathrm{Hess}_v).$$
Thus, by \eqref{Eq-w},

$$\Gamma_2(u,v)=-\mathrm{Ric}(\nabla u,\nabla v)+\mathrm{Hess}_f(\nabla u,\nabla v)-(\mathrm{Hess}_u,\mathrm{Hess}_v),$$
which proves \eqref{BE}.

\cqfd

\bigskip

Finally, we want to explain why off- and on-diagonal (Gaussian) estimates \eqref{UEL}, \eqref{DUEL} are equivalent for a generalised non-negative Schr\"{o}dinger operator $\mathcal{L}=\nabla^*\nabla+\mathcal{R}$, $\mathcal{R}\in L^1_{loc}$. Let us state this as a theorem:

\begin{Thm}\label{off-diago}

Let $(M,\mu)$ be a complete, non-compact weighted Riemannian manifold, and let $\mathcal{L}=\nabla^*\nabla +\mathcal{R}$, $\mathcal{R}\in L^1_{loc}$ be a non-negative generalised Schr\"{o}dinger operator acting on sections of $E\to M$. Then, the on-diagonal estimate \eqref{DUEL} and the Gaussian off-diagonal estimate \eqref{UEL} for the heat kernel of $\mathcal{L}$ are equivalent.

\end{Thm}
The proof relies on a general principle, according to which the validity of the Davies-Gaffney estimates for the heat operator of $\mathcal{L}$ (or, equivalently, speed of propagation $1$ for the associated wave equation), is enough to pass from on-diagonal estimate for the heat kernel, to Gaussian off-diagonal estimates. See \cite{Sik}. Let us introduce the Davies-Gaffney estimate: for every $\xi_k$, $k=1,2$ smooth sections of $E$, that coincide with the zero section respectively outside of the geodesic ball $B(x_k,r_k)$,

\begin{equation}\label{GaDa}
|\langle e^{-t\mathcal{L}}\xi_1,\xi_2\rangle| \leq Ce^{-\frac{r^2}{4t}}||\xi_1||_2||\xi_2||_2,\,\forall t>0,
\end{equation}
for all $0\leq  r<d(x_1,x_2)-(r_1+r_2)$. Thus, the results of \cite{Sik} have the consequence that \eqref{DUEL} together with \eqref{GaDa} imply \eqref{UEL}. Thus, in order to prove Theorem \ref{off-diago}, it is enough to prove that \eqref{GaDa} holds. For this, we follow a strategy developped first in \cite{Sik}, and improved later on in \cite{CS}:

\begin{proof}

Following the argument at the end of the proof of \cite[Theorem 3.3]{CS}, we see that it is enough to prove \eqref{GaDa} for $\mathcal{L}=\nabla^*\nabla=\bar{\Delta}_\mu$, the weighted rough Laplacian. For this purpose, let us take $\zeta\in C^\infty(M)$ such that $|\nabla \zeta|\leq \kappa$ for some real number $\kappa$. Take $\xi_0$ a smooth compactly supported  section, let $\xi_t:=e^{-t\mathcal{L}}\xi_0$, and define the energy

$$E(t)=\int_M|\xi_t|^2e^{\zeta}d\mu.$$
Then the proof of \cite[Theorem 6]{Sik} or \cite[Theorem 3.3]{CS} shows that in order to prove \eqref{GaDa}, it is enough to prove the estimate:

\begin{equation}\label{En}
E'(t)\leq \frac{\kappa^2E(t)}{2}.
\end{equation}
Thus, let now show \eqref{En}. We compute:

$$
\begin{array}{rcl}  
\frac{E'(t)}{2}&=&-\int_M(\nabla^*\nabla \xi_t,\xi_t)e^{\zeta}d\mu\\\\
&=&-\int_M(\nabla \xi_t,\nabla_t(e^\zeta\xi_t)d\mu\\\\
&=&-\int_M|\nabla \xi_t|^2e^\zeta d\mu-\int_M(\nabla \xi_t,d\zeta\otimes\xi_t)e^\zeta d\mu\\\\
&\leq& -\int_M|\nabla \xi_t|^2e^\zeta d\mu+\int_M|\nabla \xi_t|^2e^\zeta d\mu+\frac{1}{4}\int_M|d\zeta|^2|\xi_t|^2e^\zeta d\mu\\\\
&=&\frac{\kappa^2}{4}E(t),
\end{array}
$$
where we have used the fact that $\nabla (f\xi)=df\otimes\xi+f\nabla \xi$, by definition of a (Koszul) connection, and the inequality $|d\zeta\otimes \xi_t|\leq |d\zeta||\xi_t|$. This proves \eqref{En}, and concludes the proof of Theorem \ref{off-diago}.

\end{proof}

\begin{center}{\bf Acknowledgments} \end{center}
TC's research was  undertaken while he was employed by the Australian National University;  TC's and AS's research was  supported by an Australian Research Council (ARC) grant DP  130101302; TC acknowledges the hospitality of the D\'epartement de Math\'ematiques et Applications at the Ecole Normale Sup\'erieure,  a member institution of PSL Research University. This research was done while BD was supported by post-doctoral fellowships at the Technion (Haifa, Israel), and at the University of British Columbia (Vancouver, Canada). BD thanks both institutions for having provided him with excellent research environments, and acknowledges the support of the Israel Research Fundation, and the Natural Sciences and Engineering Research Council of Canada.

\end{document}